\documentclass{amsart}
\usepackage[margin=1in]{geometry}
\usepackage{color,amssymb, comment, mathrsfs,tikz,upgreek,contour, soul, colonequals, mathdots,fancyvrb}
\usepackage[colorlinks, linkcolor={blue}]{hyperref}
\usepackage{tikz-cd}
\usepackage{cleveref}
\usetikzlibrary{cd}
\usetikzlibrary{fit,shapes.geometric}
\usetikzlibrary{matrix}
\usetikzlibrary{arrows}
\usepackage{enumitem}
\usepackage{mathtools}
\usepackage[all]{xy}
\usepackage[mathlines,pagewise]{lineno}
\allowdisplaybreaks
\newtheorem{theorem}{Theorem}[section]
\newtheorem{lemma}[theorem]{Lemma}
\newtheorem{proposition}[theorem]{Proposition}
\newtheorem{corollary}[theorem]{Corollary}
\theoremstyle{definition}
\newtheorem*{ack}{Acknowledgements}

\usepackage[only,llbracket,rrbracket,llparenthesis,rrparenthesis]{stmaryrd} 
\usepackage{accsupp}

\theoremstyle{definition}
\newtheorem{definition}[theorem]{Definition}
\newtheorem{question}[theorem]{Question}

\newtheorem{example}[theorem]{Example}

\newtheorem{notation}[theorem]{Notation}

\theoremstyle{remark}
\newtheorem{remark}[theorem]{Remark}

\newtheorem{chunk}[theorem]{}

\numberwithin{equation}{section}

\newcommand{\syz}{\Omega}

\newcommand{\D}{\mathrm{D}}
\newcommand{\E}{\mathrm{E}}

\newcommand{\kk}{\Bbbk}

\newcommand{\xra}{\xrightarrow}
\newcommand{\id}{\operatorname{id}}

\newcommand{\Ext}{\operatorname{Ext}}

\newcommand{\Hom}{\operatorname{Hom}}

\newcommand{\RHom}{\operatorname{\textbf{R}Hom}}

\newcommand{\Ass}{\operatorname{Ass}}
\newcommand{\Coker}{\operatorname{Coker}}

\newcommand{\Tor}{\operatorname{Tor}}
\newcommand{\Tot}{\operatorname{Tot}}

\newcommand{\q}{\mathbf{q}}
\newcommand{\m}{\mathfrak{m}}
\newcommand{\x}{\mathbf{x}}
\newcommand{\depth}{\operatorname{depth}}
\newcommand{\grade}{\operatorname{grade}}
\newcommand{\rfd}{\operatorname{Rfd}}

\newcommand{\thick}{\operatorname{thick}}
\newcommand{\cidim}{\operatorname{CI-dim}}

\makeatletter
\DeclareFontEncoding{LS2}{}{\@noaccents}
\makeatother
\DeclareFontSubstitution{LS2}{stix}{m}{n}

\DeclareSymbolFont{largesymbolsstix}{LS2}{stixex}{m}{n}

\DeclareMathDelimiter{\lbrbrak}{\mathopen}{largesymbolsstix}{"EE}{largesymbolsstix}{"14}
\DeclareMathDelimiter{\rbrbrak}{\mathclose}{largesymbolsstix}{"EF}{largesymbolsstix}{"15}

\crefname{diagram}{diagram}{diagrams}
\crefname{diagram}{Diagram}{Diagrams}
\creflabelformat{diagram}{(#1) #2 #3}

\newcommand{\hsup}{\operatorname{hsup}}
\newcommand{\p}{\mathfrak{p}}
\newcommand{\Spec}{\operatorname{Spec}}
\newcommand{\pd}{\operatorname{pd}}
\newcommand{\hinf}{\operatorname{hinf}}
\newcommand{\qid}{\operatorname{qid}}
\newcommand{\cqpd}[2]{{#1}\!\operatorname{-qpd}_{#2}}
\newcommand{\cqid}[2]{{#1}\!\operatorname{-qid}_{#2}}
\newcommand{\B}[1]{\mathcal{B}_{#1}}
\newcommand{\A}[1]{\mathcal{A}_{#1}}
\newcommand{\qpd}[1]{\operatorname{qpd}_{#1}}
\renewcommand{\H}{\mathrm{H}}
\newcommand{\cpd}[2]{{#1}\!\operatorname{-pd}_{#2}}
\newcommand{\cid}[2]{{#1}\!\operatorname{-id}_{#2}}
\renewcommand{\b}{\bullet}
\newcommand{\cone}{\operatorname{cone}}
\newcommand{\Tr}{\operatorname{Tr}}
\newcommand{\lotimes}{\otimes^{\mathbf{L}}}
\newcommand{\Supp}{\operatorname{Supp}}

\keywords{Homological dimensions, Auslander classes, Bass classes, C-projectives, C-injectives, semidualizing modules, Auslander-Buchsbaum formula, Bass' formula, Ischebeck's formula, depth formula, dependency formula, quasi-projective dimension, quasi-injective dimension, quasi-projective resolution, quasi-injective resolution, vanishing of Tor, vanishing of Ext, Auslander-Reiten conjecture}

\subjclass[2020]{13D02, 13D05, 13D07, 13H10, 18G15, 18G20, 18G25.}

\author[Dey]{Souvik Dey}
\address[Souvik Dey]{Department of Mathematics \\ University of Arkansas \\ Fayetteville, AR 72701, U.S.A}
\email[]{souvikd@uark.edu}

\author[Ferraro]{Luigi Ferraro}
\address[Luigi Ferraro]{School of Mathematical and Statistical Sciences \\ University of Texas Rio Grande Valley \\ Edinburg, TX 78539, U.S.A}
\email[]{luigi.ferraro@utrgv.edu}

\author[Gheibi]{Mohsen Gheibi}
\address[Mohsen Gheibi]{Department of Mathematics \\ Florida A\&M University \\ Tallahassee, FL 32307, U.S.A}
\email[]{mohsen.gheibi@famu.edu}

\begin{document}
\title{Quasi-homological dimensions with respect to semidualizing modules}

\begin{abstract}
Gheibi, Jorgensen and Takahashi recently introduced the quasi-projective dimension of a module over commutative Noetherian rings, a homological invariant extending the classic projective dimension of a module, and Gheibi later developed the dual notion of quasi-injective dimension. Takahashi and White in 2010 introduced the projective and injective dimension of a module with respect to a semidualizing module, which likewise generalize their classic counterparts. In this paper we unify and extend these theories by defining and studying the quasi-projective and quasi-injective dimension of a module with respect to a semidualizing module. We establish several results generalizing classic formulae such as the Auslander-Buchsbaum formula, Bass' formula, Ischebeck's formula, Auslander's depth formula and Jorgensen's dependency formula. Furthermore, we prove a special case of the Auslander-Reiten conjecture and investigate rigidity properties of Ext and Tor.
\end{abstract}
\maketitle
%    Let $C$ be a semidualizing module. In this paper, we generalize quasi-projective and quasi-injective dimensions in the context of $C$-quasi homological dimensions. We recover various results known for finite $C$-projective, $C$-injective, quasi-projective, and quasi-injective dimensions such as the Auslander-Buchsbaum and Bass formulae, and Auslander's depth formula. Moreover, we prove Jorgensen's dependency formula for modules of finite quasi-projective dimension and generalize it for $C$-quasi-projective dimension.

\section{Introduction}

 The study of homological dimensions has long been central to understanding the structure and properties of modules over commutative Noetherian rings. Classic dimensions such as projective, injective, and flat dimensions provide fundamental invariants that capture subtle algebraic and homological behaviors of modules over such rings. For example, a ring is regular if and only if every module has finite projective dimension. To investigate modules over singular rings, several homological invariants that extend the concept of projective dimension have been introduced. Auslander and Bridger \cite{AusBrid} defined and explored the Gorenstein dimension. Subsequently, Avramov, Gasharov, and Peeva \cite{cidim} developed the complete intersection dimension.

 In recent decades, generalizations of these notions, especially through the lens of semidualizing modules have emerged as powerful tools to investigate module categories beyond the traditional setting. Various refinements and generalizations of these dimensions have been introduced, particularly in connection with Gorenstein homological dimensions; see for example \cite{holmj}. Recently, in another direction, quasi-homological dimensions were introduced and studied by Gheibi, Jorgensen and Takahashi in \cite{qpd} and by Gheibi in \cite{qid} as generalizations of the classic projective and injective dimensions respectively. These homological invariants provide information about modules and the base ring, as it is the case for the classic projective and injective dimension. The results proved in \cite{qpd} seem to suggest that modules of finite quasi-projective dimension behave homologically like modules over complete intersections, or, more generally, modules of finite complete intersection dimension.

Grothendieck \cite{groth} pioneered the concept of dualizing modules, which have become fundamental tools in the study of cohomology theories within algebraic geometry. Building on this foundation, Foxby \cite{foxbyC}, Vasconcelos \cite{vas}, and Golod \cite{golodC} independently introduced the notion of semidualizing modules. In \cite{TakWhite}, Takahashi and White defined the notion of $C$-projective dimension of a module, where $C$ is a semidualizing module. Takahashi and White show that there is a remarkable connection between modules with finite projective dimension and modules with finite $C$-projective dimension, see \cite[Theorem 2.11]{TakWhite}.

In this work, we take a further step and introduce quasi-homological dimensions with respect to a semidualizing module. These dimensions generalize the notions of quasi-projective and quasi-injective dimensions by replacing projective (resp. injective) modules with $C$-projective (resp. $C$-injective) modules where $C$ is a semidualizing module. The resulting theory unifies and extends known results on quasi homological dimensions and on homological dimensions with respect to semidualizing modules. In particular, we show that several homological formulae such as the Auslander–Buchsbaum formula (\Cref{thm:AB}), the Bass' formula (\Cref{Bassformula}) and Ischebeck's formula (\Cref{thm:supext} and \Cref{thm:supExt}) persist in the $C$-quasi setting under appropriate hypotheses. We point out that Ischebeck's formula for modules of finite quasi-projective and quasi-injective dimensions was recently settled in \cite{qpdIschebeck}.

\begin{comment}
\begin{theorem} Let $R$ be a local ring, and let $M$ be a finitely generated nonzero $R$-module. 
\begin{enumerate}
    \item If $\cqpd{C}{R}M<\infty$, then $\cqpd{C}{R}M=\depth R-\depth M$.
    \item If $\cqid{C}RM < \infty$ and $\Tor^R_{>0}(C,M)=0$, then $\cqid{C}RM = \depth R$.
\end{enumerate}
\end{theorem}
\end{comment}

Furthermore, we establish semidualizing versions of Auslander's depth formula (\Cref{thm:cpdDepth}) and Jorgensen’s dependency formula (\Cref{cor:DepForCpd}) for modules of finite $C$-projective dimension and extend them to the context of finite $C$-quasi-projective dimension (\Cref{thm:DepthformulaQPD} and \Cref{cor:DepFormQPD} respectively).

\begin{comment}
For nonzero finitely generated $R$-modules $M$ and $N$, assume $\Tor^R_{>>0}(M,N)=0$. According to \cite{Jorgensen}, $M$ and $N$ satisfy in {\em dependency formula} if 
$$\sup\{i\mid\Tor^R_i(M,N)\neq 0\}=\sup\{\depth R_{\p}-\depth_{R_\p}M_\p-\depth_{R_p}N_\p\mid\p\in \Spec R\}.$$
In \cite{Jorgensen}, it is shown that $M$ and $N$ satisfy in dependency formula provided one of them has finite complete intersection dimension. The following result extends the class of $R$-modules satisfying the dependency formula.

\begin{theorem}
Let $R$ be a local ring, and let $M$, $N$ be nonzero finitely generated $R$-modules such that $\Tor^R_{>>0}(M,N)=0$. Assume that $\cqpd{C}{R}M<\infty$ and $M\in\B{C}(R)$, where $\A{C}(R)$ and $\B{C}(R)$ are respectively the Auslander and Bass classes of $R$ with respect to $C$. Then $M$ and $N$ satisfy the dependency formula.    
\end{theorem}

We also prove Ischebeck's formula for modules of finite $C$-quasi homological dimensions that was settled recently in \cite{} for modules of finite quasi-projective and quasi-injective dimensions.

\begin{theorem} Let $R$ be a local ring, and let $M$, $N$ be finitely generated nonzero $R$-modules such that $\Ext^{>>0}_R(M,N)=0$. Assume either of the following conditions hold.
\begin{enumerate}
    \item One has $M\in \A{C}(R)$, $\Tor^R_{>0}(C,N)=0$ and $\cqid{C}{R}N<\infty$.
    \item One has $M\in \B{C}(R)$, $\Ext^{>0}_R(C,N)=0$,  $\cqpd{C}{R}M$ is finite.
\end{enumerate} Then  one has $$\sup\{ i| \Ext^i_R(M,N)\neq 0\} =  \depth R - \depth M.$$
\end{theorem}
\end{comment}

The paper is organized as follows. In Section 2, we review necessary background information and fix notation. Section 3 investigates properties of quasi-projective and $C$-projective dimensions, including their behavior under extensions and direct sums. Moreover, we prove Auslander's depth formula and Jorgensen's dependency formula in this setting. In Section 4 we define the $C$-quasi-projective and $C$-quasi-injective dimensions and prove transfer formulae analogous to the transfer formulae proved by Takahashi and White in \cite[Theorem 2.11]{TakWhite} in the non-quasi setting. In Section 5 we further study the properties of the $C$-quasi-projective dimension generalizing several of the properties that were investigated in \cite{qpd} in the non-semidualizing case. It is in this section that we prove the Auslander-Buchsbaum formula for modules of finite $C$-quasi-projective dimension. We also provide a result concerning the rigidity of Ext and Tor for modules of finite $C$-quasi-projective dimension, see \Cref{ExtTor}. Section 6 contains several applications regarding the $C$-quasi-projective dimension such as the aforementioned Auslander's depth formula, Ischebeck's formula for modules of finite $C$-quasi-projective dimension and Jorgensen's dependency formula. Additionally, we prove a special case of the Auslander-Reiten conjecture (\Cref{thm:ARconj}). We also provide a characterization of Gorenstein rings (\Cref{thm:KeriTotu}), generalizing (for modules) a result of Sather-Wagstaff and Totushek \cite[Theorem 3.2]{KeriTotu} which answered a question of Takahashi and White \cite[Question 5.4]{TakWhite}. Section 7 concerns several applications of the $C$-quasi-injective dimension, most of them generalizations of results proved in \cite{qid}, including Bass' formula and Ischebeck's formula for modules of finite $C$-quasi-injective dimension. We also provide a result concering the rigidity of Ext for modules of finite $C$-quasi-injective dimension, see \Cref{Vext}. Finally, we show that if a ring with a dualizing complex admits a finitely generated module of finite $C$-quasi-projective dimension and finite injective dimension, or finite $C$-quasi-injective dimension and finite projective dimension, then this forces the ring to be Cohen-Macaulay and $C$ to be dualizing (\Cref{0.8}), generalizing \cite[Corollary 6.21]{qpd} and recovering \cite[Corollary 4.3]{qid} for rings with a dualizing complex.

\begin{ack}
The authors thanks Majid R. Zargar for pointing out \Cref{cid}. Souvik Dey was partly supported by the Charles University Research Center program No.UNCE/24/SCI/022 and a grant GA \v{C}R 23-05148S from the Czech Science Foundation. Part of this work was done when the first author visited the second author at the University of Texas Rio Grande Valley, and he is thankful for their hospitality.  Luigi Ferraro was partly supported by the Simons Foundation grant MPS-TSM-00007849. All of the work done in this paper took place when the first author was a research scientist at the Department of Algebra of Charles University, Prague, and he is very grateful for the outstanding atmosphere fostered by the department. 
\end{ack}

\section{Background and notation}

In this section, we recall necessary definitions and preliminary results from the literature. Throughout, $R$ is a commutative Noetherian ring with unitary element.

\begin{chunk}[{\bf Complexes}]\label{notations}
Let $$X_\b= (\cdots\xrightarrow{\partial_{i+2}}X_{i+1}\xrightarrow{\partial_{i+1}}X_i \xrightarrow{\partial_{i}} X_{i-1} \xrightarrow{\partial_{i-1}} \cdots)$$ be a complex of $R$-modules. We define the {\em supremum}, {\em infimum}, {\em homological supremum} and {\em homological infimum} of $X$ by
$$
\begin{cases}
\sup X_\b=\sup\{i\in\mathbb{Z}\mid X_i\ne0\},\\
\inf X_\b=\inf\{i\in \mathbb{Z}\mid X_i\ne0\},
\end{cases}
\qquad
\begin{cases}
\hsup X_\b=\sup\{i\in\mathbb{Z}\mid\H_i(X_\b)\ne0\},\\
\hinf X_\b=\inf\{i\in\mathbb{Z}\mid\H_i(X_\b)\ne0\}.
\end{cases}
$$
We say that $X_\b$ is \emph{bounded}, if $\sup X_\b-\inf X_\b <\infty$. We say that $X_\b$ is \emph{bounded below} if $\inf X_\b>-\infty$ and $X_\b$ is {\em bounded above} if $\sup X_\b<\infty$. 
Note that if $X_\b$ satisfies $X_i=0$ for all 
$i\in\mathbb Z$, then $\sup X_\b =-\infty$, $\inf X_\b=\infty$.
For an integer $j$, the complex $\Sigma^jX_\b$ is defined by $(\Sigma^j X_\b)_i = X_{i-j}$ and $\partial^{\Sigma^jX_\b}_i=(-1)^j\partial^X_{i-j}$ for all $i$. 
\end{chunk}

\begin{chunk}[\bf Derived Category]
    The \textit{derived category} $\mathrm{D}(R)$ is
the category of $R$-complexes localized at the class of all quasi-isomorphisms. We use $\simeq$ to denote isomorphisms in $\mathrm{D}(R)$. The full subcategory of homologically bounded above, bounded below, and bounded complexes are denoted by $\mathrm{D}_{\sqsubset}(R)$, $\mathrm{D}_{\sqsupset}(R)$, and $\mathrm{D}_b(R)$ respectively. We denote $\mathrm{D}_{b}^f(R)$  the full subcategory of homologically bounded complexes with finitely generated homology modules.
\end{chunk} 

\begin{chunk}[\bf Thick Subcategories]
A \emph{thick subcategory} $\mathcal{T}$ of $\D(R)$ is a nonempty full subcategory such that:
\begin{enumerate}
\item $\mathcal{T}$ is closed under isomorphisms in $\D(R)$.
\item If $X\in\mathcal{T}$, then $\Sigma X\in\mathcal{T}$.
\item In any exact triangle, if two objects are in $\mathcal{T}$, so is the third.
\item Every direct summand of an object $\mathcal{T}$ is also in $\mathcal{T}$.
\end{enumerate}
If $M$ is an object in $\D(R)$, then the intersection of thick subcategories containing $M$ is also a thick subcategory. We refer to this intersection as the thick subcategory generated by $M$ and denote it by $\thick_{R}M$. We point out that $\thick_{R}R$ is the subcategory of perfect complexes.
\end{chunk}

\begin{chunk}[\bf Semidualizing modules]
    A finitely generated $R$-module $C$ is called a {\em semidualizing} $R$-module if 
    \begin{enumerate}
\item The natural homothety map $R\rightarrow\Hom_R(C,C)$ is an isomorphism.
\item $\Ext_R^i(C,C)=0$ for all $i>0$.
    \end{enumerate}

    Let $C$ be a semidualizing $R$-module. Then one has $\Supp_R(C)=\Supp(R)$ and $\depth_RC=\depth R$; see  \cite[Proposition 2.2.3 and Theorem 2.2.6]{KeriSD}.
\end{chunk}

\begin{chunk}[\bf Auslander and Bass classes] Let $C$ be a semidualizing $R$-module.

The {\em Auslander class} $\A{C}(R)$ is the class of $R$-modules $M$ satisfying in the following conditions.
\begin{enumerate}
    \item The natural map $M \longrightarrow \Hom_R(C,C\otimes_RM)$ is an isomorphism.
    \item One has $\Tor^R_{>0}(C,M)=0=\Ext^{>0}_R(C,C\otimes_RM)$.
\end{enumerate}

The {\em Bass class} $\B{C}(R)$ is the class of $R$-modules $M$ satisfying in the following conditions.
\begin{enumerate}
    \item The evaluation map $C\otimes_R\Hom_R(C,M) \longrightarrow M$ is an isomorphism.
    \item One has $\Ext_R^{>0}(C,M)=0=\Tor_{>0}^R(C,\Hom_R(C,M))$.
\end{enumerate}
\end{chunk}

\begin{chunk}[\bf $C$-projective and $C$-injective modules]
    Let $C$ be a semidualizing module and let $P$ be a projective $R$-module. The module $P\otimes_RC$ is called a $C$-{\em projective}  $R$-module. For an $R$-module $M$, a $C$-projective resolution of $M$ is an exact complex $$\dots \to C\otimes_RP_1 \to C\otimes_RP_0 \to M \to 0,$$ where $P_i$ are projective $R$-modules. We say $\cpd{C}{R}M<\infty$ if $M$ admits a bounded $C$-projective resolution and we say that $\cpd{C}{R}M=n$ if the smallest $C$-projective resolution of $M$ has length $n$.

    Let $I$ be an injective $R$-module. The module $\Hom_R(C,I)$ is called a $C$-{\em injective} $R$-module. For an $R$-module $M$, a $C$-injective resolution of $M$ is an exact complex 
    $$0\to M \to \Hom_R(C,I_0) \to \Hom_R(C,I_1) \to \dots,$$ where the $I_i$'s are injective $R$-modules.  We say $\cid{C}{R}M<\infty$ if $M$ admits a bounded $C$-injective resolution. and we say that $\cid{C}{R}M=n$ if the smallest $C$-injective resolution of $M$ has length $n$.

    We refer the reader to \cite{TakWhite} for details about $C$-projective and $C$-injective dimensions of modules and to \cite{totushek} for the complex case.
\end{chunk}

\section{Quasi-projective and $C$-projective dimension}
In this section, we collect new results about quasi-projective and $C$-projective dimensions to be generalized, or used, in later sections for the $C$-quasi-projective dimension, and the modules are not necessarily finitely generated. Unless otherwise specified, all modules will be over a commutative Noetherian ring $R$, not necessarily local.

\begin{notation}
Let $M$ be an $R$-module and let $n$ be a positive integer. We denote by $M^{\oplus n}$ the direct sum
\[
\underbrace{M\oplus\cdots\oplus M}_{n\;\mathrm{times}}.
\]
The module $M^{\oplus0}$ is the zero module.
\end{notation}

We recall the definitions of quasi-projective and quasi-injective dimension from \cite[Definition 3.1]{qpd} and \cite[Definition 2.2]{qid} respectively. Let $R$ be a ring and let $M$ be an $R$-module.

A {\em quasi-projective resolution} of $M$ over $R$ is a bounded below complex $P_\b$ of projective $R$-modules such that for all $i\ge\inf P_\b$ there exist non-negative integers 
$a_i$, not all zero, such that $\H_i(P_\b)\cong M^{\oplus a_i}$.
We define the {\em quasi-projective dimension} of $M$  to be
$$
\qpd{R}M=
\inf\{\sup P_\b-\hsup P_\b\mid\text{$P_\b$ is a bounded quasi-projective resolution of $M$}\},
$$ and $\qpd{R}M=-\infty$ if $M=0$.

A {\em quasi-injective resolution} of $M$ over $R$ is a bounded above complex $I_\b$ of injective $R$-modules such that for all $i\leq \sup I_\b$ there exist non-negative integers $b_i$, not all zero, such that $\H_i(I_\b)\cong M^{\oplus b_i}$. We define the {\em quasi-injective dimension} of  $M$ to be
$$
\qid_R M=
\inf\{\hinf I_\b-\inf I_\b\mid\text{$I$ is a bounded quasi-injective resolution of $M$}\},
$$ and $\qid_R M=-\infty$ if $M=0$.

One has $\qpd{R}M=\infty$ or $\qid_R M=\infty$ if and only if $M$ does not admit a bounded quasi-projective or quasi-injective resolution.

The next proposition will be used to show that over local rings the quasi projective dimension of a module remains unchanged if one adds a projective summand. One should also compare it to \cite[Proposition 3.3(4)]{qpd}.

\begin{proposition}\label{sesqpd} Let $J$ be a projective module. If there is an exact sequence $0\to J \xrightarrow{j} M \to N \to 0$,
then $\qpd{R} (N)\le \sup\{1, \qpd{R}(M)\}$.

\end{proposition}    
%\textcolor{red}{add bullets}
\begin{proof} We may assume that $M$ has finite quasi-projective dimension and $N\neq0$. Let $(P_{\bullet},\partial^P_{\bullet})$ be a bounded quasi projective resolution of $M$ such that $\qpd{R}M=\sup P_\b-\hsup P_\b$. So $H_i(P_{\bullet})\cong M^{\oplus a_i}$, where all but finitely many $a_i$ s are zero. Let $G_i=J^{\oplus a_i}$ (where $G_i=0$ if $a_i=0$) and consider the bounded complex 
\[
G_{\bullet}=(\cdots \to G_{i+1}\xrightarrow{0} G_i\xrightarrow{0} G_{i-1}\to \cdots )
\]
with homologies $\H_i(G_{\bullet})=G_i=J^{\oplus a_i}$.
By projectivity there is a map $G_i\rightarrow Z_i(P_\b)$ making the following diagram commute
\[
\begin{tikzcd}
                       & Z_i(P_{\bullet}) \arrow[d, two heads] \\
G_i \arrow[r,"j^{\oplus a_i}"'] \arrow[dashed, ru] & \H_i(P_\bullet)         
\end{tikzcd}
\]
Composing the map $G_i\to Z_i(P_{\bullet})$ with the inclusion $Z_i(P_{\bullet})\to P_i$ results in a map $\alpha_i:G_i \to P_i$ such that $\text{Im} (\alpha_i) \subseteq Z_i(P_{\bullet})$. Hence, $\partial_i^P\circ \alpha_i=0=\alpha_{i-1}\circ 0$. This shows that $\alpha: G\to P$ is a chain map and for all $g\in G_i$, $[\alpha_i(g)]=j^{\oplus a_i}(g)$ in $\H_i(P_{\bullet})\cong M^{\oplus a_i}$. It follows that the induced map in homology $\H_i(\alpha): \H_i(G_{\bullet})=J^{\oplus a_i}\to \H_i(P_{\bullet})$ coincides with the map $j^{\oplus a_i}: J^{\oplus a_i}\to M^{\oplus a_i}$, which is injective. The short exact sequence of complexes $0\to P_{\bullet} \to \cone(\alpha) \to \sum^1 G_{\bullet}\to 0$ induces the following long exact sequence of homologies  
\begin{equation}\label{eq:les}
\H_{i+1}(P_{\bullet})\to \H_{i+1}(\cone(\alpha))\to \H_i(G_{\bullet})\xrightarrow{j^{\oplus a_i}} \H_i(P_{\bullet})\to \H_i(\cone(\alpha))\to \H_{i-1}(G_{\bullet})\xrightarrow{j^{\oplus a_{i-1}}} \H_{i-1}(P_{\bullet}),
\end{equation}
so the injectivity of $j$ gives us the short exact sequences $0\to J^{\oplus a_i}\xrightarrow{j^{\oplus a_i}} M^{\oplus a_i}\to \H_i(\cone(\alpha))\to 0$.  Hence, $\H_i(\cone(\alpha))\cong N^{\oplus a_i}$.  Since $G_{\bullet}$ and $P_{\bullet}$ are bounded complexes, so is $\cone(\alpha)$, therefore it is a finite quasi-projective resolution of $N$.  

Now we show that $\qpd{R}(N)\leq\sup\{1,\qpd{R}(M)\}$. If $i>\hsup P_\bullet$, then $\H_i(G_\bullet)=\H_{i+1}(P_\bullet)=0$, and therefore it follows from \eqref{eq:les} that $\H_{i+1}(\cone(\alpha))=0$. Moreover, since $\H_i(P_\bullet)=0$ one has the following short exact sequence
\[
0\rightarrow \H_i(\cone(\alpha))\rightarrow\H_{i-1}(G_\bullet)\xra{j^{a_{i-1}}}\H_{i-1}(P_\bullet).
\]
If $a_{i-1}=0$, then it follows from the previous short exact sequence that $\H_i(\cone(\alpha))=0$, since $H_{i-1}(G_\bullet)\cong J^{a_{i-1}}=0$. If $a_{i-1}\neq0$, then $\H_i(\cone(\alpha))=0$ by the injectivity of $j^{a_{i-1}}$. This shows that $\H_{i}(\cone(\alpha))=0$ for $i>\hsup P_{\bullet}$. If $i=\hsup P_\bullet$, then one has the following exact sequence
\[
0\rightarrow\H_i(G_\bullet)\xra{j^{a_i}}\H_i(P_\bullet)\rightarrow \H_i(\cone(\alpha)).
\]
If $\H_i(\cone(\alpha))=0$, then $j$ would be an isomorphism, which is a contradiction, therefore $\H_i(\cone(\alpha))\neq0$. This proves that $\hsup P_\bullet=\hsup\cone(\alpha)$. It follows from the definition of mapping cone and from the definition of the complex $G_\bullet$ that
\[
\sup(\cone(\alpha))=\sup\{\hsup P_\bullet+1,\sup P_\bullet\}.
\]
Therefore
\[
\sup(\cone(\alpha))-\hsup(\cone(\alpha))=\sup\{1,\sup P_\bullet-\hsup P_\bullet\}.
\]
This shows that $\qpd{R}(N)\leq\sup\{1,\qpd{R}(M)\}$.
\end{proof} 

\begin{corollary}\label{stable} Let $J$ be a projective module and $M$ any module. Then, $\qpd{R}(M\oplus J)$ is finite if and only if $\qpd{R}(M)$ is finite. If $R$ is local, and $M$ and $J$ are finitely generated, then $\qpd{R}(M\oplus J)=\qpd{R}(M)$. 
\end{corollary} 

\begin{proof}  The first assertion follows from the exact sequence $0\to J \to M\oplus J\to M\to 0$, Proposition \ref{sesqpd} and \cite[Proposition 3.3(3)]{qpd}. 

Now assume $R$ is local, and $J\neq 0$. Due to first part, we know $\qpd{R}(M)$ and $\qpd{R}(M\oplus J)$ are simultaneously finite. An application of \cite[Corollary 4.5(2)(a)]{qpd} yields the desired result. 
\end{proof}

The following Proposition can be proved as \cite[Proposition 3.5(1) or (2)]{qpd}, and its proof is therefore omitted. It will be used throughout the paper to study how quasi-projective dimension behaves under base change.

\begin{proposition}\label{prop:qpdBaseChange} Let $R\to S$ be a ring homomorphism and $M$ an $R$-module. If  $\Tor^R_{>0}(M,S)=0$, then $$\qpd{S}(M\otimes_R S)\le \qpd{R} M.$$  
Moreover, if $P_\b$ is a quasi-projective resolution of $M$, then $P_\b\otimes_RS$ is a quasi-projective resolution of $M\otimes_RS$.
\end{proposition}

\begin{definition}
Following \cite{Jorgensen}, we say that finitely generated modules $M$ and $N$ over a local ring $R$ satisfy the \emph{dependency formula} if $\Tor^R_{>>0}(M,N)=0$ implies
\[
\sup\{i\mid\Tor_i^R(M,N)\neq0\}=\sup\{\depth_{R_\p}R_\p-\depth_{R_\p}M_\p-\depth_{R_\p}N_\p\mid\p\in\Spec(R)\}.
\]
\end{definition}

\begin{theorem}\label{thm:DepFor}
    Let $R$ be a local ring and let $M$, $N$ be nonzero finitely generated $R$-modules. Assume that $\qpd{R}M<\infty$ and $\Tor^R_{>>0}(M,N)=0$. Then $M$ and $N$ satisfy  the dependency formula.
\end{theorem}
\begin{proof}
    Let $F_\b=(0\to F_n \overset{\partial_n}\to F_{n-1} \to \dots \overset{\partial_1}\to F_0 \to 0)$ be a minimal quasi-free resolution of $M$, see \cite[Proposition 4.1]{qpd} for a proof of its existence. Set $Z_i=\ker (\partial_i)$, $B_i=\mathrm{Im} (\partial_i)$, $C_i=F_i/B_{i+1}$, and $s=\hsup F_\b$. It follows that $\pd_RC_s<\infty$. First we show for all $\p \in \Spec R$, $\depth_{R_\p}(C_s)_\p\geq \depth_{R_\p}M_\p$.
    Indeed, we even claim the stronger statement that $\depth_{R_\p} (C_j)_{\p}\ge \depth_{R_\p} M_{\p}$ for all $j\ge 0$ and $\p \in \Spec(R)$. We prove this by induction: For the base case $j=0$, we see that $(C_0)_\p=(H_0(F_\b))_\p$, hence $\depth_{R_\p} (C_0)_\p\ge \depth_{R_\p} M_\p$.  Now let $j\ge 0$ and assume the claim is true for $(C_j)_\p$.  Consider the following exact sequences 
    $$\begin{cases}
    %0\to (Z_j)_\p\to (F_j)_\p\to (B_j)_\p\to0\\
    %0\to (B_{j+1})_\p\to (Z_j)_\p\to \H_j(F)_\p\to0 \\
    0\to (B_{j+1})_\p \to (F_j)_p\to (C_j)_\p\to 0\\
    0\to \H_{j+1}(F_\b)_\p \to (C_{j+1})_\p \to (B_{j+1})_\p \to 0,
    \end{cases}$$
    then
    \begin{align*}
\depth_{R_\p}(C_{j+1})_\p&\geq\inf\{\depth_{R_\p} H_{j+1}(F_\b)_\p,\depth_{R_\p}(B_{j+1})_\p\}\\
&\geq\inf\{\depth_{R_\p} M_\p,\depth_{R_\p} R_\p,\depth_{R_\p}(C_j)_{\p}+1\}\\
&\geq\inf\{\depth_{R_\p} M_\p,\depth_{R_\p} R_\p, \depth_{R_\p} M_\p+1\}\\
&\geq\depth_{R_\p} M_\p,
    \end{align*}
where the first inequality follows from the depth lemma applied to the second short exact sequence, the second inequality follows from $\depth_{R_\p} \H_j(F_\b)_\p\geq\depth_{R_\p} M_\p$ for all $j$ and from applying the depth lemma to the first short exact sequence, the third inequality is by the inductive hypothesis  and the last one follows from the Auslander-Buchsbaum formula for $M_\p$ which has finite quasi-projective dimension over $R_{\p}$. 
    
By using \cite[Proposition 2.6]{Jorgensen} and \cite[Theorem 4.11]{qpd}, it is enough to show that if $\depth_{R_\p}\Omega^i M_\p + \depth_{R_p}N_\p\geq \depth R_\p$ for all $\p \in \Spec R$, then $\Tor^R_{>0}(\Omega^i M,N)=0$. By \cite[Proposition 3.3(4)]{qpd} $ \qpd{R}\Omega^iM<\infty$ for all $i\geq0$, therefore it suffices to prove the claim for $i=0$. Assume $\depth_{R_\p}M_\p + \depth_{R_p}N_\p\geq \depth R_\p$  for all $\p \in \Spec R$. Since $\depth_{R_\p}(C_s)_\p \geq \depth_{R_\p}M_\p$  for all $\p \in \Spec R$ and $\pd_R C_s<\infty$, we have $\Tor^R_{>0}(C_s,N)=0$ by \cite[Theorem 2.7]{Jorgensen}. Let $t=\sup\{i|\Tor^R_i(M,N)\neq 0\}$. If $t>0$, the exact sequence $0 \to \H_s(F) \to C_s \to B_s \to 0$ shows that $\Tor^R_{t+1}(B_s,N)\cong \Tor^R_t(\H_s(F_\b),N)\neq 0$. On the other hand, by using the exact sequences 
    $$\begin{cases}
    0\to Z_j\to F_j\to B_j\to0\\
    0\to B_{j+1}\to Z_j\to \H_j(F_\b)\to0 
    \end{cases}$$
    and induction argument, one has $\sup\{i|\Tor^R_i(B_j,N)\neq 0\}\leq t$ for all $j$. Indeed: For $j=0$, $B_0=0$, and $B_1$ is isomorphic to a syzygy (up to free summand) of $H_0(F_\b)\cong M^{\oplus a_0}$. Hence, $\Tor^R_{>t}(H_0(F_\b),N)=0$ gives $\Tor^R_{>t}(B_1,N)=0$. If $\Tor^R_{>t}(B_j, N)=0$, then the first short exact sequence above implies $\Tor^R_{>t}(Z_j,N)=0$, and then the second short exact sequence gives $\Tor^R_{>t}(B_{j+1},N)=0$. This is a contradiction for $j=s$, and therefore $t>0$ cannot be true, i.e., $t=0$.
\end{proof}

\begin{example}
Let $l,m,n$ be integers $\geq2$ and let $i,j\in\{1,\ldots,n\}$. Let $\kk$ be a field. Consider the following ring and the following modules
\[
R=\frac{\kk[[x_1,\ldots,x_l,y_1,\ldots,y_m,z_1,\ldots,z_n]]}{(x_1,\ldots,x_l)^2+(y_1,\ldots,y_m)^2},\quad M=\frac{R}{(x_1,\ldots,x_l,z_1,\ldots,z_i)},\quad N=\frac{R}{(y_1,\ldots,y_m,z_j,\ldots,z_n)}.
\]
Then by \cite[Example 4.1]{Jorgensen}, it follows that 
\[
\cidim_RM=\infty,\quad\cidim_RN=\infty,\quad \Tor^R_{>>0}(M,N)=0,
\]
therefore one cannot use \cite[Theorem 2.2]{Jorgensen} to deduce that $M$ and $N$ satisfy the dependency formula. We show that $M$ has finite quasi-projective dimension, and therefore by \Cref{thm:DepFor} $M$ and $N$ satisfy the dependency formula.  Indeed, the map
\[
A\colonequals\frac{\kk[[x_1,\ldots,x_l]]}{(x_1,\ldots,x_l)^2}\longrightarrow R
\]
is flat. Since by \cite[Proposition 3.6(1)]{qpd} $\qpd{A}A/(x_1,\ldots,x_n)<\infty$, the flatness of the map above implies $\qpd{R}R/(x_1,\ldots,x_l)<\infty$ by \cite[Proposition 3.5(1)]{qpd}. By the regularity of the sequence $z_1,\ldots, z_i$ on both $R$ and $R/(x_1,\ldots,x_l)$, it follows that $\qpd{R}M<\infty$ by \cite[Proposition 3.5(2) and (3)]{qpd}.
\end{example}

The following Lemma will be used in the proof of \Cref{cor:cqpd=cpd}. It shows that the $C$-projective dimension can be computed locally.
\begin{lemma}\label{lem:cpdloc}
Let $C$ be a semidualizing $R$-module, and $M$ a finitely generated $R$-module, then
\[
\cpd{C}{R}M=\sup\{\cpd{C_\p}{R_\p}M_\p\mid\p\in\Spec R\}.
\]
\end{lemma}
\begin{proof}
By \cite[Theorem 2.11c]{TakWhite}, the first and last equalities below hold (as $C_\p$ is also semidualzing over $R_\p$)
\begin{align*}
\cpd{C}{R}M&=\pd_R\Hom_R(C,M)\\
&=\sup\{\pd_{R_\p}\Hom_{R_\p}(C_\p,M_\p)\mid\p\in\Spec R\}\\
&=\sup\{\cpd{C_\p}{R_\p}M_\p\mid\p\in\Spec R\}.\qedhere
\end{align*}

\end{proof}
\begin{comment}
Let $n$ be a nonnegative integer. By \cite[Theorem 3.2(a) and Theorem 4.1]{TakWhite}, it follows that
\begin{align*}
\cpd{C}{R}M\leq n &\Leftrightarrow \Ext_R^{n+1}(\Hom_R(C,M),\Hom_R(C,-))=0\\
&\Leftrightarrow\Ext_{R_\p}^{n+1}(\Hom_{R_\p}(C_\p,M_\p),\Hom_{R_p}(C_\p,-))=0\quad\forall\p\in\Spec R\\
&\Leftrightarrow\cpd{C_\p}{R_\p}M_\p\leq n\quad\forall\p\in\Spec R,
\end{align*}
yielding the desired equality. \textcolor{red}{finish this proof or rewrite it using the transfer, add a cid version, do it for complexes}
\end{comment}

The next result proves the Depth formula for modules of finite projective dimension with respect to a semidualizing module.

\begin{theorem}\label{thm:cpdDepth}
Let $R$ be a local ring. Let $M$ and $N$ be finitely generated $R$-modules and $C$ a semidualizing $R$-module. If the following conditions are satisfied
\begin{enumerate}
\item $\cpd{C}{R}N<\infty$,
\item $M\in\A{C}(R)$,
\item $\Tor_{>0}^R(M,N)=0$,
\end{enumerate}
then
\[
\depth(M\otimes_RN)=\depth M+\depth N-\depth R.
\]
\end{theorem}
\begin{proof}
By \cite[Theorem 2.21(c)]{TakWhite} $\pd_R\Hom_R(C,N)<\infty$, and by \cite[Corollary 2.9(a)]{TakWhite} $N\in\B{C}(R)$. By \cite[Lemma 3.1.13(c)]{KeriSD} it follows that
\[
\Tor^R_i(M,N)\cong\Tor^R_i(M\otimes_RC,\Hom_R(C,N)),\quad\forall i,
\]
therefore the classic depth formula applies yielding
\[
\depth(M\otimes_RC\otimes_R\Hom_R(C,N))=\depth(M\otimes_RC)+\depth\Hom_R(C,N)-\depth R.
\]
We notice that since $N\in\B{C}(R)$ it follows that $C\otimes_R\Hom_R(C,N)\cong N$ and therefore the left-hand side of the previous display reduces to $\depth(M\otimes_RN)$. Moreover $\depth\Hom_R(C,N)=\depth N$ by \cite[Lemma 3.9]{AiTak}. It remains to show that $\depth(M\otimes_RC)=\depth M$. Since $M\in\A{C}(R)$ it follows that $\Ext_R^{>0}(C,M\otimes_RC)=0$, and therefore by \cite[Lemma 3.9]{AiTak} $\depth\Hom_R(C,M\otimes_RC)=\depth(M\otimes_RC)$, but $\Hom_R(C,M\otimes_RC)\cong M$ since $M\in\A{C}(R)$, concluding the proof.
\end{proof}
The next result proves the dependency formula for modules of finite projective dimension with respect to a semidualizing module.

\begin{corollary}\label{cor:DepForCpd}
    Let $R$ be a local ring, $C$ a semidualizing $R$-module and let $M$, $N$ be nonzero finitely generated $R$-modules. Assume that
    \begin{enumerate}
    \item $\cpd{C}{R}M<\infty$,
    \item $N\in\A{C}(R)$,
    \item $\Tor^R_{>>0}(M,N)=0$.
    \end{enumerate}
    Then $M$ and $N$ satisfy the dependency formula.
\end{corollary}

\begin{proof}
By \cite[Corollary 2.9(a)]{TakWhite}, $M\in\B{C}(R)$, therefore by \cite[Lemma 3.1.13(c)]{KeriSD}
\[
\Tor^R_i(M,N)\cong\Tor_R^i(\Hom_R(C,M),C\otimes_RN)\quad\forall i.
\]
By \cite[Theorem 2.11(c)]{TakWhite} $\pd_R\Hom_R(C,M)<\infty$, therefore by \cite[Theorem 2.2]{Jorgensen}
\[
\sup\{i\mid\Tor^R_i(M,N)\neq 0\}=\sup\{\depth R_{\p}-\depth_{R_\p}\Hom_{R_\p}(C_\p,M_\p)-\depth_{R_p}(C_\p\otimes_{R_\p}N_\p)\mid\p\in \Spec R\}.
\]
By \cite[Corollary 3.4.2]{KeriSD}, $C_\p$ is a semidualizing $R_\p$-module. By \cite[Proposition 3.4.8]{KeriSD}, $M_\p\in\B{C_\p}(R_\p)$. It follows by \cite[Lemma 3.9]{AiTak} that $\depth_{R_\p}\Hom_{R_\p}(C_\p,M_\p)=\depth_{R_\p}M_\p$. By \cite[Proposition 3.4.7]{KeriSD} $N_\p\in\A{C_\p}(R_\p)$. \Cref{thm:cpdDepth} and \cite[Theorem 2.2.6(c)]{KeriSD} imply $\depth_{R_\p}(C_\p\otimes_{R_\p}N_\p)=\depth_{R_\p}N_\p$, proving the dependency formula.
\end{proof}

\begin{remark}
For the definition and properties of $\cpd{C}{R}$ and $\cid{C}{R}$ for complexes that will be used in the next proof and throughout the paper, we refer the reader to \cite{totushek}.
\end{remark}

The next Lemma will be used later in the paper. It can be proved as \cite[Lemma 4.1]{GheibiZargar}, we provide a different proof.

\begin{lemma}\label{lem:thick}
Let $C$ be a semidualizing $R$-module and let $X_\b$ be a bounded $R$-complex with finitely generated homology.
\begin{enumerate}
\item If $\cid{C}{R}\H_i(X_\b)<\infty$ for all $i$, then $\cid{C}{R}X_\b<\infty$.
\item If $\cpd{C}{R}\H_i(X_\b)<\infty$ for all $i$, then $\cpd{C}{R}X_\b<\infty$.
\end{enumerate}
\end{lemma} 

\begin{proof}
\begin{enumerate}
\item By \cite[Theorem 2.11(b)]{TakWhite}, it follows that $\id_R\H_i(X_\b)\otimes_RC<\infty$. By \cite[Lemma 2.5]{qpd}, there is the following spectral sequence
\[
\Tor_j^R(\H_i(X_\b),C)\Rightarrow\H_{i+j}(X_\b\lotimes_RC).
\]
Since, by \cite[Corollary 2.9(b)]{TakWhite}, the modules $\H_i(X_\b)$ are in $\A{C}(R)$, it follows that the spectral sequence above collapses yielding
\[
\H_i(X_\b)\otimes_RC\cong\H_{i}(X_\b\lotimes_RC).
\]
Hence $\id_R\H_i(X_\b\lotimes_RC)<\infty$. It follows from \cite[Lemma 4.1]{GheibiZargar} that $\id_RX_\b\lotimes_RC<\infty$, which implies, by \cite[Definition 3.1(iii)]{totushek}, that $\cid{C}{R}X_\b<\infty$.

\item From \cite[Corollary 2.10(a)]{TakWhite} it follows that $\cpd{C}{R}H_i(X_\b)<\infty$ implies $H_i(X_\b)\in \thick_{R}(C)$. Therefore \cite[3.10]{DGI} implies $X_\b\in \thick_{R}(C)$. Since $\RHom_R(C,C)\cong R$, we get $\RHom_R(C,X_\b)\in \thick_{R}(R)$, hence $\pd_R \RHom_R(C,X_\b)<\infty$. By \cite[Definition 3.1(i)]{totushek} we get $\cpd{C}{R}X_\b<\infty$.  \qedhere
\end{enumerate}
\end{proof}

\section{Transfer formulae}

In this section we define the quasi-projective and quasi-injective dimensions with respect to a semidualizing module and prove formulae relating them to the classic quasi-projective and quasi-injective dimensions.

\begin{definition}\label{def:cqpd} Let $C$ be a semidualizing $R$-module. An $R$-module $M$ is said to have \emph{finite $C$-quasi-projective dimension} if there exists a bounded complex $P_{\bullet}$ of projective $R$-modules such that $P_{\bullet}\otimes_R C$ is not acyclic and all the homologies are a finite direct sum of copies of $M$ (or zero). Such a complex $P_{\bullet}$ is said to be a $C$-quasi-projective resolution of $M$. The $C$-quasi-projective dimension of $M$ is defined as
\begin{align*}
    \cqpd{C}{R}M&=\inf\{\sup (P_{\bullet} \otimes_R C) - \hsup (P_{\bullet}\otimes_R C) \mid P_{\bullet} \text{ is a } C\text{-quasi-projective resolution of } M\}.
\end{align*}
The $C$-quasi-projective dimension of the zero module is set to be $-\infty$.

\end{definition}
\begin{lemma}\label{cquasilemma}
    Let $C$ be a semidualizing $R$-module and $X_\b$ be a homologically nontrivial bounded complex of flat $R$-modules. Then, $\sup X_\b=\sup(X_\b\otimes_R C)$ and  $\hsup X_\b=\hsup(X_\b\otimes_R C)$.
\end{lemma}
\begin{proof} For the equality of sup, it is enough to observe that if $Y$ is a nonzero flat $R$-module then $Y\otimes_R C$ is also nonzero, which follows from the isomorphism $\Hom_R(C,C)\cong R$ and the tensor evaluation isomorphism 
\[
\Hom_R(C,C\otimes_R Y)\cong \Hom_R(C,C)\otimes_R Y\cong Y.
\]
For the claim on hsup, we first note that by definition of semidualizing module $\RHom_R(C,C)\cong R$ and $C\otimes_R X_\b\cong C\lotimes_R X_\b$. Moreover, the isomorphism $\RHom_R(C, C\otimes_R X_\b)\cong X_\b$ follows from \cite[A.4.23]{gbook}. Hence, $C\otimes_R X_\b$ is also homologically nontrivial. Since $\Supp_R(C)=\Spec(R)$, our claim now follows by \cite[A.4.6, A.8.7]{gbook}. 
\end{proof}
As a straightforward corollary of \Cref{cquasilemma}, we get the following
\begin{corollary}\label{cor:qpdP}
Let $C$ be a semidualizing $R$-module and $M$ an $R$-module, then
\begin{align*}
    \cqpd{C}{R}M&=\inf\{\sup P_{\bullet} - \hsup P_{\bullet} \mid P_{\bullet} \text{ is a } C\text{-quasi-projective resolution of } M\}.
\end{align*}
\end{corollary}

\begin{remark}\label{rmk:2ndDef}
Let $Q_\b$ be a complex of $C$-projective modules. Then 
\[
\Hom_R(C,Q_\bullet)\otimes_RC\cong Q_\bullet.
\]
Indeed, due to the naturality of the evaluation map, it suffices to check the isomorphism componentwise. The latter isomorphism follows since projective modules belong to the Auslander class of $C$.
This shows that $Q_\bullet$ is isomorphic to a complex of the form $P_\bullet\otimes_RC$ where $P_\bullet$ is a complex of projective modules. Therefore in the previous definition it is not restrictive to consider complexes of the form $P_\bullet\otimes_RC$ with $P_\bullet$ a complex of projective modules instead of considering all complexes of $C$-projective modules.
\end{remark}

%\textcolor{red}{define cqpd by not assuming that the C-proj res is induced by a complex of projectives. Prove that in the finitely generated case it is induced by a perfect complex}

\begin{remark}
It follows from \cite[Corollary 2.10(a)]{TakWhite} that $\cqpd{C}{R}M\leq\cpd{C}{R}M$. Moreover, by \cite[Corollary 2.9(a)]{TakWhite} if $\cpd{C}{R}M<\infty$, then $M\in\B{C}(R)$. Therefore, every module of finite $C$-projective dimension is a module of finite $C$-quasi projective dimension belonging to the Bass class. Since several results in this paper are about modules of finite $C$-quasi projective dimension in $\B{C}(R)$, this provides a class of examples satisfying the hypotheses of our results.
\end{remark}

\begin{example}
Several results in this paper require a module $M$ with $M\in\B{C}(R)$. In this example we provide a module $M$ over a Cohen-Macaulay ring with $M\in\B{C}(R), \cqpd{C}{R}M<\infty, \cpd{C}{R}M=\infty$ and $C\not\cong R$.

Let $\kk$ be a field, let $A=\kk[[t]]$ and $Q=A[x,y]/(x,y)^2$. Let $D=\Hom_A(Q,A)$, then $D$ is a semidualizing $Q$-module such that $D\not\cong Q$, see for example \cite[Example 2.3.1]{KeriSD}.
Let $R=Q/t^2Q, M=D/tD$ and $C=D/t^2D$.

\begin{itemize}
\item We first prove that $C\not\cong R$. Indeed, we first notice that $\pd_QD=\infty$ since a semidualizing module of finite projective dimension must be isomorphic to the ring itself. It follows from \cite[Lemma 1.3.5]{BrunsHerzog} that $\pd_R C=\pd_{Q/t^2Q} D/t^2D=\infty$, therefore $C\not\cong R$.

\item We now show that $M\in\B{C}(R)$. Since $\pd_QQ/tQ<\infty$, it follows from \cite[Theorem 2.11(a)]{TakWhite} that $\cpd{D}{Q}M<\infty$. Thus, by \cite[Theorem 2.9(1)]{TakWhite}, the module $M$ is in $\B{D}(Q)$. We wish to apply \cite[Proposition 3.4.6(b)]{KeriSD} to deduce that $M\in\B{C}(R)$. In order to do so we need $R\in\A{D}(Q)$. This follows from \cite[Proposition 3.1.9]{KeriSD} since $R=Q/t^2Q$ and $t^2$ is $Q$-regular.

\item Next we show that $\cpd{C}{R}M=\infty$. Since $(0:_Rt)=tR$, one can explicitly write the minimal $R$-free resolution of $R/tR$ yielding $\pd_RR/tR=\infty$. By \cite[Theorem 2.11(a)]{TakWhite}, one has that $\cpd{C}{R}C/tC=\infty$. Now it remains to observe that $C/tC\cong M$ as $R$-modules.

\item Finally, we prove that $\cqpd{C}{R}M<\infty$. This follows from the following chain of (in)equalities
\begin{align*}
\cqpd{C}{R}M&=\cqpd{C}{R}C/tC\\
&\leq\cpd{D}{Q}C/tC\\
&=\cpd{D}{Q}D/tD\\
&=\pd_QQ/tQ\\
&<\infty,
\end{align*}
where the first inequality follows from \Cref{deform}, the second equality is true since $C/tC\cong D/tD$ as $Q$-modules, the last equality is an application of \cite[Theorem 2.11(a)]{TakWhite}, and the last inequality holds by the $Q$-regularity of $t$.
\end{itemize}
\end{example}
Theorem \ref{thm:cqpd=qpd} is a quasi-projective version of \cite[Theorem 2.11(c)]{TakWhite}. We first prove a preliminary lemma.

\begin{lemma}\label{lem:ExtSpectral}
Let $P_\bullet$ be a bounded complex of projective modules, then there is the following convergent spectral sequence 
\[
\Ext_R^p(C,\H_q(P_\bullet\otimes_R C))\Rightarrow\H_{q-p}(P_\bullet).
\] 
\end{lemma}
\begin{proof}
Let $F_\bullet$ be a free resolution of $C$. Then the double complex $\Hom_R(F_\bullet,P_\bullet\otimes_R C)$ induces a second quadrant spectral sequence $\Ext_R^p(C,\H_q(P_\bullet\otimes_R C))\Rightarrow\H_{q-p}(\Hom_R(F_\bullet,P_\bullet\otimes_RC))$. It remains to show that $\Hom_R(F,P_\bullet\otimes_RC)$ is quasi-isomorphic to $P_\bullet$. One has the following isomorphisms 
\begin{align*}
\Hom_R(F,P_\bullet\otimes_RC) &\cong \RHom_R(C,P_\bullet\otimes_RC)\\
&\cong \RHom_R(C,P_\bullet\overset{L}\otimes_RC)\\
&\cong \RHom_R(C,C)\overset{L}\otimes_RP_\bullet\\
&\cong P_\bullet,
\end{align*} where the second isomorphism exists as $P_\bullet$ is a perfect complex, and the third isomorphism is by \cite[A.4.23]{gbook}.
\end{proof} 
% \textcolor{blue}{
% \begin{corollary}
%     Let $P_\bullet$ be a perfect complex. Then one has $\sup(P_\bullet)=\sup(P_\bullet\otimes_RC)$ and $\hsup(P_\bullet)=\hsup(P_\bullet\otimes_R C)$.
% \end{corollary}
% \begin{proof}
%     The first equality is obvious. We show $\hsup(P_\bullet)=\hsup(P_\bullet\otimes_R C)$. Set $h=\hsup(P_\bullet)$ and $k=\hsup (P_\bullet\otimes_R C)$.
%     Consider the spectral sequence in Lemma \ref{lem:ExtSpectral}. The maps on $\E_2$ page are of bidegree $(2,1)$. Thus $\E^{0,k}_2\cong \E^{0,k}_\infty$. Since $\E^{0,h}_2 \cong \Hom_R(C,\H_k(P_\bullet\otimes_RC))\neq 0$, it is a nonzero subqoutient of $\H_k(P_\bullet)$. Therefore $\H_k(P_\bullet)\neq 0$ and so $k\leq h$. Next, consider the truncated complex 
%     $$P'_\bullet\otimes_RC=(0\to P_n\otimes_R C\to \dots \to P_k\otimes_RC \to 0).$$ This complex is exact except at $k$. Then the isomorphism of complexes $\Hom_R(C,P'_\bullet \otimes_R C)\cong P'_\bullet$ shows that $P'_\bullet$ is exact except possibly at $k$. Therefore $h\leq k$.
% \end{proof} }

\begin{theorem}\label{thm:cqpd=qpd} 
Let $M\in \B{C}(R)$, then
\[
\cqpd{C}{R}M<\infty\iff \qpd{R}\Hom_R(C,M)<\infty.
\]
Moreover 
\begin{enumerate}
\item $\cqpd{C}{R}M=\qpd{R}\Hom_R(C,M)$.
\item $P_\b$ is a $C$-quasi-projective resolution of $M$ if and only if it is a quasi-projective resolution of $\Hom_R(C,M)$.
\end{enumerate}
\end{theorem}

\begin{proof}
We first assume that $\qpd{R}\Hom_R(C,M)<\infty$. In which case there is a bounded complex of projective modules $P_\bullet$ such that $\H_i(P_\bullet)\cong\Hom_R(C,M)^{\oplus a_i}$, for some integer $a_i$. Since $M$ is in the Bass class, it follows that $\Tor_{>0}^R(C,\Hom_R(C,M))=0$ and $C\otimes_R\Hom_R(C,M)\cong M$. Therefore $P_\bullet$ is a $C$-quasi projective resolution of $M$. Indeed, the following spectral sequence, whose existence was proved in \cite[Lemma 2.5]{qpd}, collapses
\[
\Tor_p^R(C,\H_q(P_\bullet))\Rightarrow\H_{p+q}(C\otimes_R P_\bullet),
\]
yielding the first isomorphism below
\[
\H_i(C\otimes_RP_\bullet)\cong C\otimes_R \H_i(P_\bullet)\cong C\otimes_R\Hom_R(C,M)^{\oplus a_i}\cong(C\otimes_R\Hom_R(C,M))^{\oplus a_i}\cong M^{\oplus a_i}.
\]

For the converse, let $P_\bullet$ be a $C$-quasi projective resolution of $M$. By \Cref{lem:ExtSpectral} one has the following spectral sequence 
\[
\Ext_R^p(C,\H_q(P_\bullet\otimes_R C))\Rightarrow\H_{q-p}(P_\bullet).
\]
Since $\H_q(P_\bullet\otimes_RC)$ is either zero, or a direct sum of copies of $M$, and since $M\in\B{C}(R)$, it follows that this spectral sequence collapses on the row $p=0$. Therefore the following isomorphism holds for all $i\geq0$
\[
\Hom_R(C,\H_i(P_\bullet\otimes_RC))\cong\H_i(P_\bullet).
\]
%\textcolor{red}{Since $P_\bullet$ is a complex of flat modules and $C$ is semidualizing, the right-hand side of the above display is isomorphic to $\H_i(P_\bullet)$, see for example \cite[Theorem 4.5.10]{LarsBook}.} 
Since $\H_i(P_\bullet\otimes_RC)$ is either zero or isomorphic to a direct sum of copies of $M$, it follows from the display above that $P_\bullet$ is a quasi-projective resolution of $\Hom_R(C,M)$. 

To prove the equality stated in the Theorem we can assume that both quantities are finite. We proved above that $P_\bullet$ is a $C$-quasi-projective resolution of $M$ if and only if it is a quasi-projective resolution of $\Hom_R(C,M)$. Now it suffices to invoke \Cref{cor:qpdP}.
\end{proof}

\begin{remark}\label{rmk:cqpdk} Let $(R,\m,\kk)$ be local, and $K_\b^R$ be the Koszul complex on a generating set of $\m$. Then, $K_\b^R$ is a bounded $C$-quasi-projective resolution of $\kk$, hence $\cqpd{C}{R} \kk\leq\mathrm{edim}\;R$, but $\kk\notin \B{C}(R)$ unless $C\cong R$.

If $M$ is a nonzero $R$-module with a periodic $C$-projective resolution, then, as in the proof of \cite[Proposition 3.6(2)]{qpd}, it follows that $\cqpd{C}{R}M=0$.
\end{remark} 

\begin{example} In \cite[Corollary 2.9(a)]{TakWhite} it is proved that a module of finite $C$-projective dimension must belong to $\B{C}(R)$. In this example we show that a module of finite $C$-quasi-projective dimension does not necessarily belong to $\B{C}(R)$.

Let $(R,\m,\kk)$ be a local artinian ring that is not Gorenstein and let $E$ be the injective hull of $\kk$. Let $M$ be the module defined by the following short exact sequence
\[
0\rightarrow \kk\rightarrow E\rightarrow M\rightarrow0.
\]
By \cite[Proposition 2.8(1)]{qid}, the residue field $\kk$ has finite quasi injective dimension. It follows from \cite[Proposition 2.4(3)]{qid} that $M$ has finite quasi injective dimension. Therefore there is a bounded complex of injective modules $I_\bullet$ whose homology modules are isomorphic to direct sums of $M$. Since an injective module is a direct sum of copies of $E$, it follows that $I_\bullet=\Hom_R(E,I_\bullet)\otimes_RE$ where $\Hom_R(E,I_\b)$ is a bounded complex of projective modules, see for example \cite[Theorem 10.3.8]{LarsBook}. Therefore $\cqpd{E}{R}M<\infty$. Since $E\in\B{E}(R)$, it follows that $M\not\in\B{E}(R)$, otherwise $\kk$ would be in $\B{E}(R)$ and this would imply that $E\cong R$.
\end{example}

\begin{theorem}\label{auslanderclass}
Let $M\in \A{C}(R)$, then
\[
\qpd{R}M<\infty\iff \cqpd{C}{R}(C\otimes_RM)<\infty.
\]
Moreover 
\begin{enumerate}
\item $\qpd{R}M=\cqpd{C}{R}(C\otimes_RM)$.
\item $P_\b$ is a quasi-projective resolution of $M$ if and only if it is a $C$-quasi-projective resolution of $C\otimes_RM$.
\end{enumerate}
\end{theorem}
\begin{proof}
We first assume that $\qpd{R}M<\infty$. Let $P_\bullet$ be a quasi-projective resolution of $M$. We claim that $P_\bullet$ is a $C$-quasi-projective resolution of $M\otimes_RC$. Indeed, since $\H_q(P_\bullet)$ is a finite direct sum of copies of $M$, and $M$ is in the Auslander class, the following spectral sequence collapses
\[
\Tor_p^R(C,\H_q(P_\bullet))\Rightarrow\H_{p+q}(C\otimes_RP_\bullet),
\]
yielding $\H_q(C\otimes_RP_\bullet)\cong C\otimes_R\H_q(P_\bullet)$. Since $\H_q(P_\bullet)$ is a finite direct sum of copies of $M$ (or zero), it follows that $\H_q(C\otimes_RP_\bullet)$ is a finite direct sum of copies of $C\otimes_RM$ (or zero).

For the converse assume that $\cqpd{C}{R}(C\otimes_RM)<\infty$. Let $P_\bullet$ be a $C$-quasi-projective resolution of $C\otimes_RM$. We claim that $P_\bullet$ is a quasi-projective resolution of $M$. Indeed, since $\H_q(C\otimes_RP_\bullet)$ is a finite direct sum of copies of $C\otimes_RM$, and $M\in\A{C}(R)$, it follows that the following spectral sequence from \Cref{lem:ExtSpectral} collapses
\[
\Ext_R^p(C,\H_q(P_\bullet\otimes_RC))\Rightarrow\H_{q-p}(P_\b),
\]
yielding the isomorphism 
\[
\Hom_R(C,\H_q(P_\bullet\otimes_RC))\cong\H_q(P_\b).
\]
It follows that $P_\bullet$ is a quasi-projective resolution of $\Hom_R(C,C\otimes_RM)$, which is isomorphic to $M$ since $M\in\A{C}(R)$. The proof of \emph{(1)} is similar to the proof of \Cref{thm:cqpd=qpd}\emph{(1)}.
\end{proof}

\begin{definition}\label{def:cqid}
Let $C$ be a semidualizing $R$-module. An $R$-module $M$ is said to have finite $C$-quasi-injective dimnesion if there exists a bounded complex $I_\b$ of injective $R$-modules such that $\Hom_R(C,I_\b)$ is not acyclic and all the homologies are finite direct sum of copies of $M$ (or zero). Such a complex $I_\b$ is said to be a \emph{$C$-quasi-injective resolution} of $M$. The \emph{$C$-quasi-injective dimension} of $M$ is defined as 
\begin{align*}
    \cqid{C}{R}M&=\inf\{\hinf (\Hom_R(C,I_\b)) - \inf (\Hom_R(C,I_\b)) \mid I_\b \text{ is a } C\text{-quasi-injective resolution of }M\},
\end{align*}
and $\cqid{C}{R}M=-\infty$ if $M=0$. 
\end{definition}

\begin{lemma}\label{cqinj} Let $C$ be a semidualizing $R$-module, and $X_\b$ be a homologically nontrivial bounded complex of injective $R$-modules. Then, $\inf X_\b=\inf \Hom_R(C,X_\b)$ and $\hinf X_\b=\hinf \Hom_R(C,X_\b)$.
\end{lemma}
\begin{proof}For the equality of inf, it is enough to observe that if $Y$ is a nonzero injective $R$-module, then $\Hom_R(C,Y)$ is also nonzero, which follows from the isomorphism $\Hom_R(C,C)\cong R$ and Hom evaluation isomorphism 
\[
C\otimes_R \Hom_R(C,Y)\cong \Hom_R(\Hom_R(C,C),Y)\cong Y.
\]
For the claim on hinf, we first note that $\RHom_R(C,C)\cong R$ and $\Hom_R(C,X_\b)\cong \RHom_R(C,X_\b)$. Moreover, the isomorphism $C\lotimes_R\Hom_R(C,X_\b)\cong X_\b$ follows from \cite[A.4.24]{gbook}. Hence, $\Hom_R(C,X_\b) $ is also homologically nontrivial. Since $\operatorname{Supp}(C)=\Spec(R)$, we are now done by \cite[A.4.15, A.8.8]{gbook}.
\end{proof}

As a straightforward corollary of \Cref{cqinj}, we get the following
\begin{corollary}\label{cor:qidP}
Let $C$ be a semidualizing $R$-module and $M$ an $R$-module, then
\[
\cqid{C}{R}M =\inf\{\hinf I_\b-\inf I_\b\mid I_\b\text{ is a }C\text{-quasi-injective resolution of }M\}.
\]
\end{corollary}

The next Lemma is proved similarly to \Cref{lem:ExtSpectral}, but instead of using  \cite[A.4.23]{gbook}, one needs to use \cite[A.4.24]{gbook}.
\begin{lemma}\label{lem:ExtSpectralInj}
Let $I_\bullet$ be a bounded complex of injective modules, then there is the following convergent spectral sequence 
\[
\Tor_p^R(C,\H_q(\Hom_R(C,I_\b)))\Rightarrow\H_{p+q}(I_\bullet).
\] 
\end{lemma}

\begin{theorem}\label{auslanderclassid}
Let $M\in \A{C}(R)$, then
\[
\cqid{C}{R}M<\infty\iff \qid_{R}(C\otimes_RM)<\infty.
\]
Moreover 
\begin{enumerate}
\item $\cqid{C}{R}M=\qid_{R}(C\otimes_RM)$.
\item $I_\b$ is a $C$-quasi-injective resolution of $M$ if and only if it is a quasi-injective resolution of $C\otimes_RM$.
\end{enumerate}
\end{theorem}

\begin{proof}
We first assume that $\cqid{C}{R}M<\infty$. Let $I_\b$ be a $C$-quasi-injective resolution of $M$. Consider the following spectral sequence, which exists by \Cref{lem:ExtSpectralInj}
\[
\Tor_p^R(C,\H_q(\Hom_R(C,I_\b)))\Rightarrow\H_{p+q}(I_\bullet).
\] 
Since the homologies of $\Hom_R(C,I_\b)$ are either zero or a direct sum of copies of $M$, this spectral sequence collapses on the row $p=0$ since $C$ and $M$ are Tor-independent because $M\in\A{C}(R)$. This shows that
\[
\H_{q}(I_\b)\cong C\otimes_R\H_q(\Hom_R(C,I_\b)).
\]
This proves that $I_\b$ is a quasi-injective resolution of $C\otimes_RM$.

The converse follows in a similar manner from the following spectral sequence (see \cite[Lemma 2.5]{qid})
\[
\Ext_R^p(C,\H_q(I_\b))\Rightarrow\H_{q-p}(\Hom_R(C,I_\b)).
\] and keeping in mind $\RHom_R(C, C\otimes_R M)\cong M$ because $M\in\A{C}(R)$.

To prove part $(1)$ one can just invoke \Cref{cor:qidP}.
\end{proof}
An analogous proof shows the following
\begin{theorem}\label{thm:bassCqidTrans}
Let $M\in \B{C}(R)$, then
\[
\cqid{C}{R}\Hom_R(C,M)<\infty\iff \qid_{R}M<\infty.
\]
Moreover 
\begin{enumerate}
\item $\cqid{C}{R}\Hom_R(C,M)= \qid_{R}M$.
\item $I_\b$ is a quasi-injective resolution of $M$ if and only if it is a $C$-quasi-injective resolution of $\Hom_R(C,M)$.
\end{enumerate}
\end{theorem}

\section{Properties of the $C$-quasi-projective dimension}
In this section we establish various properties of the $C$-quasi-projective dimension of modules, many of which generalize analogous results on the quasi-projective dimension or on the $C$-projective dimension.

The next result is a generalization of \cite[Remark 3.2(4)]{qpd} and will be used later in the paper.

\begin{lemma}\label{lem:hinf=inf}
Let $M$ be a nonzero $R$-module, let $C$ be a semidualizing $R$-module and $P_\b$ a $C$-quasi-projective resolution of $M$. Then there exists a $C$-quasi-projective resolution $P'_\b$ of $M$, with $\H_i(P'_\b\otimes_RC)=\H_i(P_\b\otimes_RC)$ for all $i\in\mathbb{Z}$ and $\hinf (P'_\b\otimes_RC)=\inf (P'_\b\otimes_RC)$.
\end{lemma}
\begin{proof}
Let $\partial_\b$ be the differential of $P_\b$. Set $u=\inf (P_\b\otimes_RC)$. Let $K\colonequals\ker(\partial_{u+1}\otimes_RC)$. If $\H_u(P_\b\otimes_RC)=0$, then one has the following short exact sequence
\[
0\rightarrow K\rightarrow P_{u+1}\otimes_RC\xra{\partial_{u+1}\otimes_RC}P_u\otimes_RC\rightarrow0.
\]
By \cite[Corollary 2.9(a)]{TakWhite}, it follows that $P_{u+1}\otimes_RC,P_u\otimes_RC\in\B{C}(R)$. By the two-out-of-three property we deduce that $K\in\B{C}(R)$, and therefore $\Ext_R^1(P_u\otimes_RC,K)=0$. This implies that the short exact sequence displayed above splits. By \cite[Theorem 2.11(c)]{TakWhite}, it follows that direct summands of $C$-projective modules are $C$-projective, and therefore $K$ is $C$-projective. As a result, there is an isomorphism $\iota:K\rightarrow Q\otimes_RC$, where $Q$ is a projective $R$-module. This implies that the complex
\[
\cdots\xra{\partial_{u+4}\otimes_RC}P_{u+3}\otimes_RC\xra{\partial_{u+3}\otimes_RC}P_{u+2}\otimes_RC\xra{(\partial_{u+2}\otimes_RC)\iota}Q\otimes_RC\rightarrow0
\]
has the same homology has $P_\b\otimes_RC$ and by \Cref{rmk:2ndDef} we are done.
\end{proof}
The next Proposition generalizes \cite[Proposition 3.3]{qpd}.
\begin{proposition} \label{prop:3.3} Let $R$ be a commutative noetherian ring and $C$ a semidualizing $R$-module.
\begin{enumerate}
\item Let $M$ be an $R$-module and $n$ a positive integer, then $\cqpd{C}{R}M^{\oplus n}=\cqpd{C}{R}M$.
\item Let $M$ and $N$ be $R$-modules, then
\[
\cqpd{C}{R}(M\oplus N)\leq\sup\{\cqpd{C}{R}M,\cqpd{C}{R}N\}.
\]
\item Let $M$ be a nonzero $R$-module and $J$ a $C$-projective module, then $\cqpd{C}{R}(M\oplus J)\leq\cqpd{C}{R}M$.
\item If there is a short exact sequence $0\to M\to P\otimes_RC\to Y\to 0$ with $Y\in \B{C}(R)$ and $P$ a projective module, then  $\cqpd{C}{R} M\le \cqpd{C}{R} Y$. 
\end{enumerate}
\end{proposition}

\begin{proof}
\begin{enumerate}

\item If $P_\b$ is a $C$-quasi-projective resolution of $M$, then $P_\b^{\oplus n}$ is a $C$-quasi-projective resolution of $M^{\oplus n}$. Conversely, if $P_\b$ is a $C$-quasi-projective resolution of $M^{\oplus n}$, then it is also a $C$-quasi-projective resolution of $M$.

\item We may assume that $M$ and $N$ are nonzero and are of finite $C$-quasi-projective dimensions. Let $P_\b,P'_\b$ be finite $C$-quasi-projective resolutions of $M$ and $N$, respectively, with
\[
\cqpd{C}{R}M=\sup (P_\b\otimes_RC)-\hsup (P_\b\otimes_RC),\quad\cqpd{C}{R}N=\sup (P'_\b\otimes_RC)-\hsup (P'_\b\otimes_RC).
\]
Therefore
\[
\H_i(P_\b\otimes_RC)\cong M^{\oplus a_i},\quad \H_j(P'_\b\otimes_RC)\cong N^{\oplus b_j}
\]
for some $a_i,b_j\geq0$. Consider the complex
\[
F_\b=\left(\bigoplus_{j\in\mathbb{Z}} \Sigma^j(P_\b\otimes_RC)^{\oplus b_j}\right)\oplus\left(\bigoplus_{i\in\mathbb{Z}}\Sigma^i (P'_\b\otimes_RC)^{\oplus a_i}\right),
\]
Then
\[
\H_k(F_\b)=(M\oplus N)^{\oplus \sum_{i+j=k}a_ib_j},
\]
and therefore $\Hom_R(C,F_\b)$ is a $C$-quasi-projective resolution of $M\oplus N$ by \Cref{rmk:2ndDef}. Moreover
\begin{align*}
\sup F_\b&=\max\{\sup (P_\b\otimes_RC)+\hsup(P'_\b\otimes_RC),\sup(P'_\b\otimes_RC)+\hsup(P_\b\otimes_RC)\},\\
\hsup F_\b &=\hsup(P_\b\otimes_RC)+\hsup(P'_\b\otimes_RC).
\end{align*}
It follows that
\begin{samepage}
\begin{align*}
\cqpd{C}{R}(M\oplus N)&\leq\sup F_\b-\hsup F_\b\\
&=\max\{\sup(P_\b\otimes_RC)-\hsup(P_\b\otimes_RC),\sup(P'_\b\otimes_RC)-\hsup(P'_\b\otimes_RC)\\
&=\max\{\cqpd{C}{R}M,\cqpd{C}{R}N\}.
\end{align*}
\end{samepage}
\item Follows from (2) by letting $N=J$.
\item We may assume $\cqpd{C}{R} Y<\infty$. Since $P\otimes_RC\in\B{C}(R)$ by \cite[Corollary 2.9(a)]{TakWhite}, by the two-out-of-three property $M\in \B{C}(R)$, hence $\Ext_R^1(C,M)=0$, so $0\to\Hom_R(C,M)\to P\to \Hom_R(C,Y)\to 0$ is exact. By Theorem \ref{thm:cqpd=qpd} and \cite[Proposition 3.3(4)]{qpd}, $\qpd{R}\Hom_R(C,M)\le \qpd{R}\Hom_R(C,Y)=\cqpd{C}{R} Y$.
By Theorem \ref{thm:cqpd=qpd} again, we get $\cqpd{C}{R} M\le \cqpd{C}{R} Y$. \qedhere
\end{enumerate}
\end{proof}
The next Corollary generalizes \Cref{sesqpd}.
\begin{corollary} Let $J$ be a $C$-projective $R$-module and let $0\to J \to M \to N \to 0$ be an exact sequence. 
If $M\in \B{C}(R)$, then $\cqpd{C}{R} (N)\le \sup\{1, \cqpd{C}{R}(M)\}$.
\end{corollary}

\begin{proof} By \cite[1.9(a) and Corollary 2.9(a)]{TakWhite}, $J,N\in \B{C}(R)$. Applying $\Hom_R(C,-)$ to the short exact sequence yields an exact sequence  $0\to \Hom_R(C,J)\to \Hom_R(C,M)\to \Hom_R(C,N)\to 0$. Since $\Hom_R(C,J)$ is projective by \cite[Theorem 2.11(c)]{TakWhite}, it follows that $\qpd{R} \Hom_R(C,N)\leq \sup\{1, \qpd{R}\Hom_R(C,M)\}$ by \Cref{sesqpd}. Now we are done by \Cref{thm:cqpd=qpd}. 
\end{proof}

The next Corollary follows immediately from \Cref{stable} and \Cref{thm:cqpd=qpd}, it studies the behavior of the $C$-quasi-projective dimension under adding a $C$-projective summand. 

\begin{corollary}
Let $C$ be a semidualizing $R$-module and $J$ a $C$-projective module. Let $M$ be an $R$-module in $\B{C}(R)$. Then $\cqpd{C}{R}(M\oplus J)$ is finite if and only if $\cqpd{C}{R}(M)$ is finite. If $R$ is local, and $M$,$J$ are finitely generated, then $\cqpd{C}{R}(M\oplus J)=\cqpd{C}{R}(M)$.
\end{corollary}

The following Proposition generalizes \cite[Proposition 3.4 and 4.1]{qpd}.

\begin{proposition}\label{minimal}
    Let $M$ be a nonzero finitely generated $R$-module of finite $C$-quasi projective dimension. 
    \begin{enumerate}
        \item There exists a perfect complex $P_\b$ such that $P_\b$ is a $C$-quasi projective resolution of $M$ and $\cqpd{C}{R}M=\sup P_\b\otimes_R C - \hsup P_\b\otimes_RC$.
        \item If $R$ is local, then $P_\b$ can be chosen to be minimal.
    \end{enumerate}
\end{proposition}
\begin{proof}
    Let $P_\b'$ be a $C$-quasi projective resolution of $M$ such that $\cqpd{C}{R}M=\sup P_\b'\otimes_R C - \hsup P_\b'\otimes_RC$. By \cite[A.4.23]{gbook}, $\RHom_R(C,C\otimes_RP_\b')\cong \RHom_R(C,C)\lotimes_RP_\b'\cong P_\b'$. Since $P_\b'\in \D_b(R)$ and $P_\b'\otimes_R C\in \D^f_b(R)$, hence $P_\b'\cong \RHom_R(C,C\otimes_RP_\b') \in \D^f_b(R)$ by \cite[A.4.4]{gbook}. By \cite[A.3.2]{gbook} there exists a perfect complex $P_\b$ and a quasi-isomorphism $P\overset{\simeq}\longrightarrow P_\b'$ which induces a quasi-isomorphism $P_\b\otimes_RC\overset{\simeq}\longrightarrow P_\b'\otimes_RC$; see \cite[A.4.1]{gbook}. By the same argument as in the proof of \cite[Proposition 3.4]{qpd}, there exists a perfect complex $P_\b''$ quasi-isomorphic to $P_\b'$ such that $\sup(P_\b'')=\sup(P_\b')$. Since $C$ has full support it follows that $\sup(P_\b''\otimes_R C)=\sup(P_\b'\otimes_R C)$. Since $P_\b''\otimes_R C$ is quasi-isomorphic to $P_\b'\otimes_RC$ they also have the same hsup. This finishes the proof of (1). 
    Part (2) follows from (1) and the same argument as \cite[Proposition 4.1]{qpd}.
\end{proof}

\begin{remark}\label{rmk:Exercise} Let $R \to S$ be a ring homomorphism. Let $X_\b$ be a bounded below chain complex of $R$-modules such that  $\Tor^R_{>0}( X_i \oplus \H_i(X_\b), S)=0$ for all $i\in \mathbb{Z}$. Then, $\H_i(X_\b\otimes_R S)\cong \H_i(X_\b)\otimes_R S$ for all $i\in \mathbb Z$.  

Indeed, this can be seen by first shifting the complex, and assuming $\inf X_\b=0$, then proceeding by induction by repeatedly breaking the complex in to cycles, boundaries and homologies and tensoring with $S$. 
\end{remark}

\begin{theorem}\label{thm:flatext}
Let $R\rightarrow S$ be a ring homomorphism of finite flat dimension and let $C$ be a semidualizing $R$-module. Let $M$ be an $R$-module with $\Tor^R_{>0}(M, S)=0$. If $P_\b$ is a $C$-quasi-projective resolution of $M$, then $P_\b\otimes_RS$ is a $(C\otimes_RS)$-quasi-projective resolution of $M\otimes_RS$. Consequently, 
\[
\cqpd{(C\otimes_RS)}{S}(M\otimes_RS)\leq\cqpd{C}{R}M.
\]

% %with $M\in\B{C}(R)$ and $\Tor^R_{>0}(M,S)=0$, then
% \[
% \cqpd{(C\otimes_RS)}{S}(M\otimes_RS)\leq\cqpd{C}{R}M.
% \]
% Moreover, if $P_\b$ is a $C$-quasi-projective resolution of $M$, then $P_\b\otimes_RS$ is a $(C\otimes_RS)$-quasi-projective resolution of $M\otimes_RS$.

\end{theorem}

\begin{proof}

Since $S$ has finite flat dimension, it follows, by \cite[Corollary 3.4.2]{KeriSD}, that the module $C\otimes_RS$ is a semidualizing $S$-module. It also follows from \cite[1.9(b)]{TakWhite} that $\Tor^R_{>0}(C,S)=0$. Since each $P_i$ is a direct summand of a free module, thus each $(P_{\bullet}\otimes_R C)_i$ is a direct summand of a direct sum of copies of $C$, consequently, $\Tor^R_{>0}( (P_{\bullet}\otimes_R C)_i, S)=0$ for all $i\in \mathbb Z$. Since the homologies of $P_{\bullet}\otimes_R C$ are direct sum of copies of $M$, we conclude by 
\Cref{rmk:Exercise} that $\H_i((P_{\bullet}\otimes_R C)\otimes_R S)\cong \H_i(P_{\bullet}\otimes_R C) \otimes_R S\cong M^{\oplus a_i}\otimes_R S$ for all $i\in \mathbb Z$ and some integers $a_i\ge 0$, where the last isomorphism is by hypothesis since $P_{\bullet}$ os a $C$-quasi-projective resolution of $M$. Since  $(P_{\bullet}\otimes_R C)\otimes_R S\cong P_{\bullet}\otimes_R (C\otimes_R S)$ and $C\otimes_RS$ is a semidualizing $S$-module, this proves the claim.   \qedhere

\end{proof}

\begin{comment}
\textcolor{red}{\Cref{basechange} should follow by inspecting the proof of \cite[3.8(2)]{ARC} and in turn will generalize that result.}

\begin{lemma}\label{basechange} Let $R\to S$ have finite flat dimension, let $X,M$ be $R$-modules such that $\Ext_R^{>0}(X,M)=0=\Tor^R_{>0}(X\oplus M, S)$. Then, $\Ext_S^{>0}(X\otimes_R S, M\otimes_R S)=0$, and $\Hom_S(X\otimes_R S, M\otimes_R S)\cong \Hom_R(X,M)\otimes_R S$. \textcolor{red}{for our purpose, we will use it with $X=C$ semidualizing ... so we may assume $\rfd_R X\leq 0$ if so needed.} 
\end{lemma}

\begin{proof} 
    
\end{proof}

\textcolor{red}{If $C$ is semidualizing $R$-module, $R\to S$ is a ring map such that $S\in \A{C}(R)$ and $M$ is an $R$-module such that $\Tor^R_{>0}(M,S)=0$}

% \textcolor{red}{in view of \cite[3.5(1),(2)]{qpd}, we ask: If we have a  ring map $R\to S$ of finite flat dimension and $R$-module $M$ satisfying $\Tor^R_{>0}(M,S)=0$, then is $\qpd{S}(M\otimes_R S)\le \qpd{R} M$ ?}  

\end{comment}

The following Corollary is an immediate consequence of \Cref{thm:flatext} and \Cref{rmk:Exercise}. 

\begin{corollary} \label{lem:regSeq} Let $M$ be an $R$-module, let $\mathbf x$ be a sequence of elements of $R$ regular on both $R$ and $M$, and let $C$ be a semidualizing $R$-module. Then 
\[\cqpd{(C/\mathbf x C)}{R/\mathbf x}M/\mathbf x M\le \cqpd{C}{R} M.
\]
Moreover, if $P_\b$ is a $C$-quasi-projective resolution of $M$, then $P_\b\otimes_RR/\mathbf{x}$ is a $C/\mathbf{x}C$-quasi-projective resolution of $M/\mathbf{x}M$ and $\hsup P_{\bullet}\otimes_R C \otimes_R R/\mathbf x=\hsup P_{\bullet}\otimes_R C$. 
\end{corollary}

% The following Proposition generalizes \cite[Proposition 3.5(3)]{qpd}.
% \begin{proposition}
% Let $R$ be a local ring and $\x$ an $R$-regular sequence of length $c$. Let $C$ be a semidualizing $R$-module. If $M$ is a finitely generated $R/\x$-module with $M\in\B{C}(R)$, then
% \[
% \cqpd{C}{R}M\leq\cqpd{(C/\x C)}{R/\x}M+c.
% \]

% \end{proposition}

% \begin{proof}
% We can assume that $M\neq0$ and that $\cqpd{(C/\x C)}{R/\x}M<\infty$. By \cite[Corollary 3.4.2]{KeriSD} the module $C/\x C$ is semidualizing over $R/\x$. By \cite[Proposition 3.4.6]{KeriSD} it follows that that $M\in\B{C/\x C}(R/\x)$ since $R/\x$ has finite flat dimension over $R$ and therefore $R/\x\in\A{C}{(R)}$. Hence, by \Cref{thm:cqpd=qpd} the inequality in the statement is equivalent to
% \[
% \qpd{R}\Hom_R(C,M)\leq\qpd{R/\x}\Hom_{R/\x}(C/\x C,M)+c.
% \]
% Since $\Hom_{R/\x}(C/\x C,M)\cong\Hom_R(C,M)$ this inequality follows from \cite[Proposition 3.5(3)]{qpd}.
% \end{proof}

\Cref{deform} aims to generalize \cite[Proposition 3.7]{qpd}.

\begin{proposition}\label{deform}
Let $\x$ be an $R$-regular sequence and $M$ an $R/\x$-module. Let $C$ be a semidualizing $R$-module. If $P_\bullet$ is a $C$-projective $R$-resolution of $M$, then $P_\b\otimes_R(R/\x)$ is a $C/\x C$-quasi-projective $R/\x$-resolution of $M$. In particular
\[
\cqpd{(C/\x C)}{R/\x}M\leq\cpd{C}{R}M.
\]
\end{proposition}
\begin{proof}
We may assume $\cpd{C}{R} M<\infty$, therefore by \cite[Corollary 2.9]{TakWhite} $M\in \B{C}(R)$. Moreover, since $R/\x$ has finite flat dimension over $R$, it follows that $R/\mathbf x\in \A{C}(R)$. Thus \cite[Proposition 3.4.6]{KeriSD} implies $M\in \B{C/\mathbf x C}(R/\mathbf x)$. Hence, by \Cref{thm:cqpd=qpd} and \cite[Theorem 2.11]{TakWhite}, the inequality in the statement is equivalent to $\qpd{R/\mathbf x R} \Hom_{R/\mathbf x R}(C/\mathbf x C, M)\le \pd_R \Hom_R(C, M)$. Since $M$ is an $R/\mathbf x R$-module it follows that $\Hom_{R/\mathbf x R}(C/\mathbf x C, M)\cong \Hom_R(C, M)$ (see \cite[Lemma 2(ii)]{matsu} for instance). Now we are done by \cite[Proposition 3.7]{qpd}.

For the second assertion, let $P_\b$ be a $C$-projective $R$-resolution of $M$, then $P_\b$ is a projective $R$-resolution of $\Hom_R(C,M)$. By \cite[Proposition 3.7]{qpd} $P_\b\otimes_RR/\x$ is a quasi-projective $R/\x$-resolution of $\Hom_R(C,M)\cong\Hom_{R/\x}(C/\x C,M)$. By \Cref{thm:cqpd=qpd} it follows that $P_\b\otimes_RR/\x$ is a $C/\x C$-quasi-projective $R/\x$-resolution of $M$.\qedhere
\begin{comment}
We need to show that the homology of $P_\b\otimes_R(R/\x)\otimes_{R/\x}(C/\x C)$ is either zero or consists of direct sums of copies of $M$. The following spectral sequences collapses since $P_\b$ is a $C$-projective resolution of $M$
\[
\Tor_p^R(\H_q(P_\b\otimes_RC),R/\x)\Rightarrow\H_{p+q}(P_\b\otimes_RC\otimes_RR/\x),
\]
\textcolor{red}{double check if this spectral sequence is correct} yielding $\Tor^R_p(M,R/\x)\cong\H_p(P_\b\otimes_RC\otimes_RR/\x)$. Now it remains to notice that
\[
\Tor_p^R(M,R/\x)\cong\H_p(M\otimes_R K^R(\x))=M^{\binom{|\x|}{p}},
\]
where $K^R(\x)$ is the Koszul complex of the sequence $\x$ over $R$ and the last equality follows since $\x M=0$.
\end{comment}

\end{proof}

We now prove the Auslander-Buchsbaum formula for modules of finite $C$-quasi-projective dimension.

\begin{theorem}\label{thm:AB}
Let $R$ be a local ring and let $M$ be a nonzero finitely generated $R$-module. If $\cqpd{C}{R}M<\infty$, then
\[
\cqpd{C}{R}M=\depth R-\depth M=\sup F_\b-\hsup F_\b,
\]
for any minimal $C$-quasi-projective resolution $F_\b$ of $M$.
\end{theorem} 
\begin{proof}
By Proposition \ref{minimal}, there exists a minimal perfect complex $F_\b$ such that $\H_i(F_\b\otimes_RC)\cong M^{\oplus{a_i}}$ for some $a_i\geq 0$ and $\cqpd{C}{R}M=\sup F_\b\otimes_R C - \hsup F_\b\otimes_RC$. Set $s=\sup F_\b\otimes_RC$ and $h=\hsup F_\b\otimes_RC$. Since $F_\b$ is minimal, by applying $\Hom_R(k,-)$ to the exact sequence $$0 \to \H_s(F_\b\otimes_RC) \to F_s\otimes_RC \to F_{s-1}\otimes_RC,$$ we get an isomorphism $\Hom_R(k,\H_s(F_\b\otimes_RC))\cong \Hom_R(k,F_s\otimes_RC)$. We note that everything till this point is true for any minimal $C$-quasi-projective resolution of $M$. First assume $\depth R=0$.  Since $\depth_RC=\depth R=0$, we get $\Hom_R(k,\H_s(F_\b\otimes_RC))\neq 0$. Therefore $\H_s(F_\b\otimes_RC)\neq 0=\depth \H_s(F_\b\otimes_R C)$ and so $s=h$. Since $\H_s(F_\b\otimes_RC)\cong M^{\oplus{a_s}}$ for some $a_s> 0$, we have $\depth_RM =0$. We also note that this argument remains true for  any minimal $C$-quasi-projective resolution of $M$.\\
Next, assume $\depth R>0$ and $\depth_RM=0$. Then the exact sequence above shows that $\H_s(F_\b\otimes_RC)=0$. Therefore $\cqpd{C}{R}M=s-h>0$. Consider the exact sequence $$0\to F_s\otimes_RC \overset{\partial_s} \to F_{s-1}\otimes_RC \overset{\partial_{s-1}}\to \dots \overset{\partial_{h+1}}\to F_h\otimes_RC \to N\to 0,$$ where $N=\Coker(\partial_{h+1})$.
There exists an inclusion $\H_h(F_\b\otimes_RC) \hookrightarrow N$. Therefore one has $\depth N=0$ and $\cpd{C}{R}N=\depth R$; see \cite[Proposition 6.4.2]{KeriSD}. Since $F_\b$ is minimal, we also have $\depth R=\cpd{C}{R}N=s-h=\cqpd{C}{R}M$, finishing the proof of this case. \\ We also note that this argument remains true for  any minimal $C$-quasi-projective resolution of $M$, i.e.,
if $F_\b'$ is another minimal $C$-quasi-projective resolution of $M$, the same argument as above shows that $\cqpd{C}{R}M=\sup F_\b'\otimes_R C - \hsup F_\b'\otimes_RC$.\\
Now assume $\depth_RM>0$. The isomorphism $\Hom_R(k,\H_s(F_\b\otimes_RC))\cong \Hom_R(k,F_s\otimes_RC)$ implies that $\depth R=\depth_RC>0$. Let $x$ be a regular element on $R$, $C$, and $M$.

Since $\depth R/xR<\depth R$ and $\depth_{R/xR} M/xM<\depth_R M$, and $F_\b/xF_\b$ is a minimal $C/xC$-quasi-projective-resolution of $M/xM$ (by \Cref{lem:regSeq}) we get by inductive hypothesis that $$\cqpd{(C/xC)}{R/(x)}M/xM=\depth R/(x)-\depth_{R/(x)} M/xM=\depth R - \depth_R M$$ and $\cqpd{(C/xC)}{R/(x)}M/xM=\sup (F_\b/xF_\b\otimes_{R/(x)}C/xC)-\hsup (F_\b/xF_\b\otimes_{R/(x)}C/xC)=s-h$, where the last equality is by \Cref{lem:regSeq}. Thus, we conclude that $\sup (F_\b\otimes_R C)-\hsup(F_\b\otimes_R C)=\depth R-\depth M$, for any minimal $C$-quasi-projective resolution $F_\b$ of $M$, finishing the inductive step and the proof. 
\end{proof} 

The next result shows that modules $M$ of finite $C$-quasi-projective dimension have uniform Ext (resp. Tor) vanishing bound with respect to all modules which are also right Ext-orthogonal (resp. Tor-orthogonal) to $C$. Before that we prove a technical lemma
\begin{lemma} \label{lem:SpectralSeq}
    Let $N$ be $R$-module and $P_\b$ a bounded below complex of projective modules.
    \begin{enumerate}
        \item If $\Tor^R_{>0}(C,N)=0$, then there exists a first quadrant spectral sequence $$\E^2_{p,q}\cong \Tor^R_p(\H_q(P_\b\otimes_RC),N)\Longrightarrow \H_{p+q}(P_\b\otimes_RC\otimes_RN).$$
        \item If $\Ext^{>0}_R(C,N)=0$, then there exists a third quadrant spectral sequence 
        $$\E^{p,q}_2\cong \Ext^p_R(\H_q(P_\b\otimes_RC),N)\Longrightarrow \H^{p+q}\Hom_R(P_\b,\Hom_R(C,N)).$$
    \end{enumerate}
\end{lemma}
\begin{proof}
 We only prove (1). Let $F_\b$ be a projective resolution of $N$. Then the double complex $(P_\b\otimes_RC)\otimes_RF_\b$ induces a first quadrant spectral sequence $$\E^2_{p,q}\cong \Tor^R_p(\H_q(P_\b\otimes_RC),N)\Longrightarrow \H_{p+q}(P_\b\otimes_RC\otimes_RF_\b).$$ It remains to observe that $P_\b\otimes_RC\otimes_RF_\b$ is quasi-isomorphic to $P_\b\otimes_R(C\otimes_RN)$. Indeed, one has isomorphisms
    \begin{align*}
(P_\b\otimes_RC)\otimes_RF_\b&\cong
(P_\b\otimes_RC)\lotimes_RN\\
&\cong (P_\b\lotimes_RC)\lotimes_RN\\
&\cong P_\b \lotimes_R(C\lotimes_RN)\\
&\cong P_\b \otimes_R(C\lotimes_RN)\\
&\cong P_\b \otimes_R(C\otimes_RN),
\end{align*} where the last isomorphism exists since $\Tor^R_{>0}(C,N)=0$ and $P_\b$ is a complex of flat modules. This finishes the proof.
\end{proof}

\begin{proposition}\label{ExtTor} 
    Let $M$, $N$ be $R$-modules, and assume $\cqpd{C}{R}M<\infty$. 
    \begin{enumerate} 
        \item Suppose $\Tor^R_{>0}(C,N)=0$. If $\Tor^R_{>> 0}(M,N)=0$, then $\Tor^R_i(M,N)=0$ for all $i>\cqpd{C}{R}M$.
        \item Suppose $\Ext^{>0}_R(C,N)=0$. If $\Ext^{>> 0}_R(M,N)=0$, then $\Ext^i_R(M,N)=0$ for all $i>\cqpd{C}{R}M$.
    \end{enumerate}
\end{proposition}
\begin{proof}
    We only prove (1). Let $P_\b$ be a $C$-quasi projective resolution of $M$ such that 
    \[
    \cqpd{C}{R}M=\sup P_\b\otimes_R C - \hsup P_\b\otimes_RC,
    \]
   by \Cref{lem:SpectralSeq} there is a first quadrant spectral sequence  $$\E^2_{p,q}\cong \Tor^R_p(\H_q(P_\b\otimes_RC),N)\Longrightarrow \H_{p+q}(P_\b\otimes_RC\otimes_RN).$$ Set $h=\hsup P_\b\otimes_R C$, $s=\sup P_\b$, and $n=\sup\{i| \Tor^i_R(M,N)\neq0\}$. Since differentials on $\E^2$ page are of bidegree $(-2,1)$, one has $\E^{\infty}_{n,h}\cong \E^2_{n,h}\neq 0$. Hence we must have $n+h\leq s$, and so that, $n\leq \cqpd{C}{R}M$. 
\end{proof}

The following Proposition refines \Cref{prop:3.3} in the local case
\begin{proposition}
Let $R$ be a local ring, and let $M$,$N$ be finitely generated $R$-modules. Let $C$ be a semidualizing $R$-module.
\begin{enumerate}
\item If $\cqpd{C}{R}M,\cqpd{C}{R}N<\infty$, then
\[
\cqpd{C}{R}(M\oplus N)=\sup\{\cqpd{C}{R}M,\cqpd{C}{R}N\}.
\] 
\item If $M\neq0, M\in\B{C}(R)$ and there is a short exact sequence $0\rightarrow N \rightarrow P\otimes_RC\rightarrow M\rightarrow0$, with $P$ a projective $R$-module, then
\[
\cqpd{C}{R} N\leq\sup\{\cqpd{C}{R}M-1,0\},
\]
and the equality holds if the right-hand side is finite.
\end{enumerate}
\end{proposition}
\begin{proof}
\begin{enumerate}
\item We may assume that both $M$ and $N$ are nonzero. The following equalities follow from \Cref{thm:AB} and \Cref{prop:3.3}(2)
\begin{align*}
\cqpd{C}{R}(M\oplus N)&=\depth R-\depth(M\oplus N)\\
&=\depth R-\inf\{\depth M,\depth N\}\\
&=\sup\{\depth R-\depth M,\depth R-\depth N\}\\
&=\sup\{\cqpd{C}{R}M,\cqpd{C}{R}N\}.
\end{align*}
\item By \Cref{prop:3.3}(4) it follows that $\cqpd{C}{R}N<\infty$, therefore we can apply \Cref{thm:AB}
\begin{align*}
\cqpd{C}{R}N&=\depth R-\depth N\\
&\leq\depth R-\inf\{\depth M+1,\depth(P\otimes_RC)\}\\
&=\sup\{\depth R-\depth M-1,0\}\\
&=\sup\{\cqpd{C}{R}M-1,0\},
\end{align*}
where the inequality follows from the depth lemma, while the third equality follows from the equality $\depth(P\otimes_RC)=\depth R$.\qedhere
\end{enumerate}
\end{proof}

We now show that the $C$-quasi-projective dimension of a module coincides with the $C$-projective dimension whenever the latter is finite.

\begin{corollary}\label{cor:cqpd=cpd}
Let $M$ be a finitely generated module over a (not necessarily local) ring $R$. Let $C$ be a semidualizing module. If $\cpd{C}{R}M<\infty$, then $\cqpd{C}{R}M=\cpd{C}{R}M$.
\end{corollary}
\begin{proof}
By definition of $C$-quasi-projective dimension it follows that $\cqpd{C}{R}M\leq\cpd{C}{R}M$. By \Cref{lem:cpdloc}, there exists $\p\in\Spec R$ such that $\cpd{C_\p}{R_\p}M_\p=\cpd{C}{R}M$. By \Cref{thm:AB} one has the following string of (in)equalities
\begin{align*}
\cpd{C}{R}M&\geq \cqpd{C}{R}M\\
&\geq\cqpd{C_\p}{R_\p}M_\p\\
&=\depth R_\p-\depth_{R_\p}M_\p\\
&=\cpd{C_\p}{R_\p}M_\p\\
&=\cpd{C}{R}M.
\end{align*}

Here, the second inequality is true because localization is an exact functor and sends $C$-projective modules to $C_{\q}$-projective modules. 
Therefore all inequalities in the above display are equalities, yielding $\cqpd{C}{R}M=\cpd{C}{R}M$.
\end{proof}

\begin{comment}
We take this opportunity to generalize the New Intersection Theorem for complexes of $C$-projective modules proved by Takahashi and White \cite[Proposition 5.2]{TakWhite}

\begin{theorem}
Let $R$ be a noetherian local ring with a semidualizing $R$-module $C$. Let $I$ be an ideal of $R$. Let $X$ be a non-exact complex
\[
X:=0\rightarrow C^{\alpha_s}\rightarrow C^{\alpha_{s-1}}\rightarrow\cdots\rightarrow C^{\alpha_1}\rightarrow C^{\alpha_0}\rightarrow0
\]
such that $\H_i(X)$ is $I$-torsion for $i\geq0$, then $s\geq\dim R-\dim R/I$.
\end{theorem}

\begin{proof}
The complex $\Hom_R(C,X)$ is non-exact, otherwise it would be split and therefore the complex $$C\otimes_R\Hom_R(C,X)\cong X$$ would be exact.

Let $\p\in\Spec R$ such that $I\not\subseteq\p$, then $X_\p$ is exact. Therefore one can split $X_\p$ into short exact sequences where all the terms belong to $\B{C_\p}(R_\p)$ since belonging to the Bass class has the two-out-of-three property. It follows that these short exact sequences remain exact if the functor $\Hom_{R_\p}(C_\p,-)$ is applied. By splicing these short exact sequences together one gets that the complex $\Hom_{R_\p}(C_\p,X_\p)$ is exact. Therefore the complex $\Hom_R(C,X)$ has $I$-torsion homology in all nonnegative degrees. The assertion of the theorem now follows from \cite[Theorem 2.2]{IntThm}.
\end{proof}
\end{comment}

The following Proposition generalizes \cite[Proposition 3.5(3)]{qpd}.
\begin{proposition}
Let $R$ be a local ring and $\x$ an $R$-regular sequence of length $c$. Let $C$ be a semidualizing $R$-module. If $M$ is a finitely generated $R/\x$-module with $M\in\B{C}(R)$, then
\[
\cqpd{C}{R}M\leq\cqpd{(C/\x C)}{R/\x}M+c.
\]

\end{proposition}

\begin{proof}
We can assume that $M\neq0$ and that $\cqpd{(C/\x C)}{R/\x}M<\infty$. By \cite[Corollary 3.4.2]{KeriSD} the module $C/\x C$ is semidualizing over $R/\x$. By \cite[Proposition 3.4.6]{KeriSD} it follows that $M\in\B{C/\x C}(R/\x)$ since $R/\x$ has finite flat dimension over $R$ and therefore $R/\x\in\A{C}{(R)}$. Hence, by \Cref{thm:cqpd=qpd} the inequality in the statement is equivalent to
\[
\qpd{R}\Hom_R(C,M)\leq\qpd{R/\x}\Hom_{R/\x}(C/\x C,M)+c.
\]
Since $\Hom_{R/\x}(C/\x C,M)\cong\Hom_R(C,M)$ this inequality follows from \cite[Proposition 3.5(3)]{qpd}.
\end{proof}

\section{Applications of the $C$-quasi-projective dimension}

\begin{definition}
Let $C$ be a semidualizing $R$-module and $M$ an $R$-module. We say that $M$ is $C$-\emph{virtually small} if
\[
\thick_RC\cap\thick_RM\neq\{0\}.
\]
\end{definition}

The next result is a semidualizing version of \cite[Proposition 3.11]{qpd}.

\begin{proposition}\label{cor:VirtSmall}
Let $M$ be a finitely generated nonzero $R$-module. If $\cqpd{C}{R}M<\infty$, then $M$ is $C$-virtually small.
\end{proposition}
\begin{proof}
By \Cref{minimal} there exists a non-acyclic perfect complex $P$ such that the homologies of $P\otimes_R C$ are finite direct sum of copies of $M$. By \cite[3.10]{DGI} we get $P\otimes_R C\in \thick_R(M)$. Also, clearly, $P\otimes_R C\in \thick_R C$, which finishes the claim. 
\end{proof}

We will use the next Lemma in one of the results below.

\begin{lemma} \label{lem:DVirtSmall} Let $R$ be a local ring, and let $Y\in \D^f_b(R)$ of finite injective dimension. If $M$ is a perfect complex such that $\thick_R(Y)\cap\thick_R(M)\neq\{0\}$, then $R$ is Gorenstein.   
\end{lemma}

\begin{proof} Let $0\neq X \in \thick_R(Y)\cap\thick_R(M)$. Then, $X$ has both finite projective and injective dimension as $M$ and $Y$ do respectively. Now we are done by \cite[Proposition 2.10]{foxflat}. 
\end{proof}

\begin{remark}\label{rmk:omegapd}
By \cite[Exercise 3.3.28(b)]{BrunsHerzog} and \cite[Corollary 2.10(a)]{TakWhite} it follows that for a module $M$ over a Cohen-Macaulay ring $R$ with dualizing module $\omega_R$, one has $\cpd{\omega_R}{R}M<\infty$ if and only if $\mathrm{id}_R\;M<\infty$. The next theorem generalizes this observation.
\end{remark}
\begin{theorem}\label{thm:omegaqpd}
Let $R$ be a local commutative ring with a dualizing module $\omega_R$. Let $M$ be an $R$-module. Then the following statements hold.
\begin{enumerate}
    \item One has $\cqpd{\omega_R}{R}M<\infty$ if and only if $\qid_RM<\infty$.
    \item One has $\cqid{\omega_R}{R}M<\infty$ if and only if $\qpd{R}{M}<\infty$.
\end{enumerate}
\end{theorem}

\begin{proof} We only prove (1), and a dual argument applies to (2).
First assume that $\cqpd{\omega_R}{R}M<\infty$. By hypothesis, there is a bounded complex of projective modules $P_{\bullet}$ such that $P_{\bullet}\otimes_R \omega_R$ is not acyclic and all the homologies are finite direct sum of copies of $M$. It follows from \cite[Remark 4.1]{IK} that $\id_R (P_{\bullet}\otimes_R \omega_R)<\infty$, hence there exists a bounded complex of injective modules $I_{\bullet}$ quasi-isomorphic to $P_{\bullet}\otimes_R\omega_R$. Then, $I_{\bullet}$ is a finite quasi-injective resolution of $M$ by definition.  

For the converse, assume that $\qid_RM<\infty$. Then there is a bounded complex $I_\bullet$ of injective modules that is not acyclic and all the homologies are finite direct sum of copies of $M$. Consider the following chain of quasi-isomorphisms
\[
I_\bullet\simeq\Hom_R(R,I_\bullet)\simeq\Hom_R(\Hom_R(\omega_R,\omega_R),I_\bullet)\simeq\omega_R\otimes_R\Hom_R(\omega_R,I_\bullet).
\]
where the last quasi-isomorphism is by \cite[4.5.13]{LarsBook}. By \cite[Remark 4.1]{IK} $\Hom_R(\omega_R,I_\bullet)$ has finite projective dimension. Let $P_\bullet$ be a finite projective resolution of $\Hom_R(\omega_R,I_\bullet)$. Then $P_\bullet$ is a $\omega_R$-quasi-projective resolution of $M$.
\end{proof}

\begin{remark} \Cref{def:cqpd} and \Cref{def:cqid} could be given for a semidualizing complex $C$, and with that, the proof of \Cref{thm:omegaqpd} would go through verbatim to establish the result with the dualizing module $\omega_R$ replaced by a dualizing complex. We postpone the study of quasi-homological dimensions of complexes with respect to semidualizing complexes to a future work.
\end{remark}

\begin{comment}
As a Corollary we deduce \cite[Corollary 4.3]{qid} for rings with a dualizing module

\begin{corollary}\label{cor:qid&pd}
Let $R$ be a local commutative ring with a dualizing module $\omega_R$. Let $M$ be a finitely generated $R$-module. If $\qid_{R}M<\infty$ and $\pd_RM<\infty$, then $R$ is Gorenstein.
\end{corollary}

\begin{proof}
By \Cref{thm:omegaqpd}, the module $M$ has finite $\omega_R$-quasi-projective dimension. Now \Cref{cor:VirtSmall} yields that $M$ is $\omega_R$-virtually small. The result now follows from \Cref{lem:DVirtSmall}.
\end{proof}

\begin{remark}
We point out that \Cref{cor:qid&pd} can be proved without assuming the existence of a dualizing module, hence providing an alternative proof of \cite[Corollary 4.3]{qid}. In order to achieve this one could define the $C$-quasi-projective dimension with respect to a semidualizing complex $C$ and prove \Cref{thm:omegaqpd} by using a dualizing complex instead of a dualizing module, with the same proof. Then the same proof as in  \Cref{cor:qid&pd} goes through by completing the ring (in view of \cite[Proposition 2.6(3)]{qid}) and thereby assuming the existence of a dualizing complex.  
\end{remark} 
\end{comment}

\begin{example}In this example we will show that  \Cref{prop:3.3}(4) does not hold if $Y\not\in\B{C}(R)$. 

Let $R$ be a ring with $\m^2=0$ which is not Gorenstein. Consider the following exact sequence
\[
0\rightarrow\Omega\rightarrow F\rightarrow E\rightarrow 0,
\]
where $F$ is free, $E$ is the injective hull of $\kk$ and $\Omega\subseteq\m F$. Let $(-)^\vee\colonequals\Hom_R(-,E)$ be the Matlis dual functor and consider the sequence
\[
0\rightarrow E^\vee\rightarrow F^\vee\rightarrow \Omega^\vee\rightarrow0.
\]
Since $\m^2=0$ it follows that $\Omega$ is a $k$-vector space and so is $\Omega^\vee$. Therefore, by \Cref{rmk:cqpdk}, one has that $\cqpd{E}{R}\Omega^\vee<\infty$. Since $F$ is free, it follows that $F^\vee$ is $E$-projective. We will show that $\cqpd{E}{R}E^\vee=\infty$. We point out that $E^\vee=R$. If $\cqpd{E}{R}R<\infty$, then $\qid_R R<\infty$ by \Cref{thm:omegaqpd}, which is a contradiction by \cite[Corollary 4.3]{qid}.
\end{example}

\Cref{prop:characteriz} generalizes the first half of \cite[Theorem 6.5]{qpd}. We recall that a finitely generated module $M$ is said to satisfy $(\widetilde S_n)$ if $\depth_{R_{\p}} M_{\p}\ge \inf\{n, \depth R_{\p}\}$ for all $\p \in \Spec(R)$. 

\begin{proposition}\label{prop:characteriz} Let $(R,\m,k)$ be local, $\depth R=t$. If there exists $n\ge t$ such that $\cqpd{C}{R} (C\otimes_R \Tr \syz^n_R k) <\infty$, then $R$ is Cohen--Macaulay, and $C$ is dualizing.      
\end{proposition}

\begin{proof} Let $M:=\Tr \syz^n_R k$. Similar to the proof of \cite[Theorem 6.5]{qpd}, $C\otimes_R M$ embeds in an $R$-module of finite $C$-projective dimension. Since $n\ge t$, hence $\syz^n_R k$ satisfies $(\widetilde S_t)$ by the depth Lemma. Moreover, $\syz^n_R k$ is also locally free on punctured spectrum. Thus, $\Ext_R^{1\le i \le t}(M, R)=0$ by \cite[Proposition 2.4(b)]{ds}. Now following the same proof as in \cite[Theorem 1.3 (3)$\implies$  (1)]{yty}, we get $\Ext^1_R(\Tr M, C)=0$. Hence, $\Ext^1_R(\syz^n_R k, C)=0$, i.e., $\Ext^{n+1}_R(k,C)=0$. Since $n+1\geq t+1>\depth C$, we get $\id_R C<\infty$ by \cite[II. Theorem 2]{rober}. Hence, $R$ is Cohen--Macaulay and $C$ is dualizing. 
\end{proof}  

\begin{corollary}
Let $R$ be local. If every $R$-module has finite $C$-quasi-projective dimension, then $R$ is an AB ring. In particular R is Gorenstein and $C\cong R$.
\end{corollary}

\begin{proof}
By \Cref{prop:characteriz} $R$ is Cohen-Macaulay and $C$ is dualizing. Therefore by \Cref{thm:omegaqpd} every module has finite quasi injective dimension. By \cite[Corollary 4.5]{qid} it follows that $R$ is AB.
\end{proof}

\begin{comment}
The next result generalizes \cite[Corollary 6.21]{qpd}.

\begin{theorem}
Let $R$ be a Cohen-Macaulay ring with a dualizing module $\omega_R$. Let $C$ be a semidualizing module. If $\cqpd{C}{R}\omega_R<\infty$, then $C\cong\omega_R$.
\end{theorem}
\begin{proof}
Since $\id_R\omega_R<\infty$ one has $\omega_R\in\B{C}(R)$, and therefore by \Cref{thm:cqpd=qpd} $\qpd{R}\Hom_R(C,\omega_R)<\infty$. By \cite[Corollary 4.1 and Corollary 4.2(a)]{TakWhite} it follows that
\[
\Ext_R^{>0}(\Hom_R(C,\omega_R),\Hom_R(C,\omega_R))\cong\Ext_R^{>0}(\omega_R,\omega_R)=0.
\]
Therefore by \cite[Theorem 6.20(1)]{qpd} $\pd_R\Hom_R(C,\omega_R)<\infty$. It follows that
\[
\id_R\Hom_R(\Hom_R(C,\omega_R),\omega_R)<\infty,
\]
but $\Hom_R(\Hom_R(C,\omega_R),\omega_R)\cong C$, showing that $C$ must be dualizing.
\end{proof}
\end{comment}

We next prove that the Depth formula holds under more general conditions than the classic ones.

\begin{theorem}\label{thm:DepthformulaQPD}
Let $R$ be a local ring. Let $M$ and $N$ be finitely generated $R$-modules and $C$ a semidualizing $R$-module. If the following conditions are satisfied
\begin{enumerate}
\item $\cqpd{C}{R}N<\infty$ and $N\in\B{C}(R)$,
\item $M\in\A{C}(R)$,
\item $\Tor_{>0}^R(M,N)=0$,
\end{enumerate}
then
\[
\depth(M\otimes_RN)=\depth M+\depth N-\depth R.
\]
\end{theorem}
\begin{proof}
By \Cref{thm:cqpd=qpd} $\qpd{R}\Hom_R(C,N)<\infty$, now one argues as in the proof of \Cref{thm:cpdDepth} by using \cite[Theorem 4.11]{qpd} instead of the classic version of the depth formula.
\end{proof}

The next result is a special case of the Auslander-Reiten conjecture. It generalizes \cite[Theorem 1.4]{qpd} and provides a semidualizing version over commutative rings of \cite[Theorem 3.3]{qpdAb}.
\begin{theorem}\label{thm:ARconj}
    Let $M$ and $C$ be $R$-modules where $C$ is semidualizing. If $\cqpd{C}{R}M<\infty$ and $\Ext_R^{\ge 2}(M,M)=0$, then $\cpd{C}{R}M<\infty$. Moreover, if $M$ is finitely generated, then
    \[
    \cpd{C}{R} M=\begin{cases}
    0\quad\mathrm{if } \Ext_R^1(M,M)=0\\
    1\quad\mathrm{if } \Ext_R^1(M,M)\neq0.
    \end{cases}
    \]
\end{theorem}

\begin{proof}
    There exists a bounded complex $P_{\bullet}$ of projective $R$-modules such that $P_{\bullet}\otimes_RC$ is not acyclic and its homologies are finite direct sum of copies of $M$. Therefore $$\Ext_R^{\ge 2}(\H_i(P_{\bullet}\otimes_RC), \H_j(P_{\bullet}\otimes_RC))=0$$ for all $i,j\in \mathbb Z$. By \cite[Tag 0GM4]{stacks-project} we get $C\otimes_R P_{\bullet} \cong \bigoplus_{i\in \mathbb Z}\Sigma^i M^{\oplus a_i}$ for some non-negative integers $a_i$, at least one of which is non-zero. Applying $\RHom_R(C,-)$ to both sides, and using $\RHom_R(C,C)\cong R$ and \cite[A.4.23]{gbook} we get $P_{\bullet}\cong \bigoplus_{i\in \mathbb Z}\Sigma^i\RHom_R(C,M)^{\oplus a_i}$. Thus, $\RHom_R(C,M)$ has finite projective dimension, so $\cpd{C}{R} M<\infty$ by \cite[Definition 3.1(1)]{totushek}. We now assume that $M$ is finitely generated and prove the last assertion. By localizing at a prime in the support of $M$ we can assume that $R$ is local. By \cite[Corollary 2.9(a)]{TakWhite} $M\in\B{C}(R)$ and therefore by \cite[Theorem 4.1 and Corollary 4.2(a)]{TakWhite}
    \[
\Ext_R^i(M,M)=\Ext_R^i(\Hom_R(C,M),\Hom_R(C,M)) \quad\forall i.
    \]
Since $\cpd{C}{R}M=\pd\Hom_R(C,M)$ by \cite[Theorem 2.11(c)]{TakWhite}, an application of \cite[Lemma 1(iii) page 154]{matsu} yields the desired result.
\end{proof}
In \cite[Question 5.4]{TakWhite}, Takahashi and White ask wheter the existence of a finitely generated module of finite $C$-projective dimension and finite $C$-injective dimension forces the ring to be Gorenstein. This question was answered by Sather-Wagstaff and Totushek in \cite[Theorem 3.2]{KeriTotu}. The next Theorem generalizes this result.
\begin{theorem}\label{thm:KeriTotu}
Let $C$ be a semidualizing $R$-module and $M$ be a finitely generated $R$-module such that $\cid{C}{R}M<\infty$ and $\cqpd{C}{R}M<\infty$, then $R_\p$ is Gorenstein for all $\p\in\Supp(M)$.
\end{theorem}
\begin{proof}
Since $\cqpd{C}{R}M<\infty$, there exists a bounded complex of projectives $P_\b$ such that $\H_i(P_\b\otimes_RC)\cong M^{a_i}$. Since $\cid{C}{R}M<\infty$ it follows by \Cref{lem:thick}(1) that $\cid{C}{R}(P_\b\otimes_RC)<\infty$. Moreover, by \cite[Proposition 3.6]{totushek} $\cpd{C}{R}(P_\b\otimes_RC)<\infty$. By \cite[Theorem 3.2]{KeriTotu} $R_\p$ is Gorenstein for all $\p\in\Supp(P_\b\otimes_RC)$. To conclude the proof it suffices to notice that $\Supp(P\otimes_R C)=\Supp(M)$. 
\end{proof}

The next theorem generalizes Ischebeck's formula \cite{ischebeck}.

\begin{theorem}\label{thm:supext}
    Let $R$ be a local ring, and let $M$, $N$ be finitely generated nonzero $R$-modules. Assume $M\in \B{C}(R)$ and $\Ext^{>>0}_R(M,N)=\Ext^{>0}_R(C,N)=0$.  
    If $\cqpd{C}{R}M$ is finite, then   $$\sup\{ i| \Ext^i_R(M,N)\neq 0\} =  \depth R - \depth M.$$
\end{theorem}
\begin{proof} 
     Set $t=\cqpd{C}{R}M$. By Proposition \ref{ExtTor}, we have $\Ext^i_R(M,N)=0$ for all $i>t$. It remains to show $\Ext^t_R(M,N)\neq 0$. 
    By Proposition \ref{minimal} there exists a minimal $C$-quasi free resolution $F_\b=(0\to F_n \to F_{n-1} \to \dots \to F_0 \to 0)$ of $M$ such that $t=\sup F_\b - \hsup F_\b$. If $t=0$, then we must have $\Hom_R(M,N)\neq 0$ because otherwise, $\Ext^i_R(M,N)=0$ for all $i\in \mathbb{Z}$ and therefore $\mathrm{grade}_R(\mathrm{Ann}_R(M),N)=\infty$, contradiction.
    Next, assume $t=1$. Then by definition, we have $\H_n(F_\b)=0\neq \H_{n-1}(F_\b)$. As before we need to show $\Ext^1_R(M,N)\neq 0$. Seeking a contradiction, we assume $\Ext^1_R(M,N)=0$. Let $I_\b$ be an injective resolution of $N$. Then the double complex $\Hom_R(F_\b\otimes_RC,I_\b)$ induces a third quadrant spectral sequence $$\E^{p,q}_2\cong \Ext^p_R(\H_q(F_\b\otimes_RC),N)\Longrightarrow \H^{p+q}(\Hom_R(F_\b\otimes_RC,I_\b)).$$ Since $\Ext^{>0}_R(C,N)=0$, the isomorphism of complexes $\Hom_R(F_\b\otimes_RC,I_\b)\cong \Hom_R(F_\b,\Hom_R(C,I_\b))$ shows that $\H^i(\Hom_R(F_\b\otimes_RC,I_\b))\cong \H^i\Hom_R(F_\b,\Hom_R(C,N))$. Since $\H_i(F_\b\otimes_RC)\cong M^{\oplus a_i}$ for some $a_i\geq 0$, and $\Ext^{>0}_R(M,N)=0$, the spectral sequence above collapses. Therefore
     we get isomorphisms \[
     \Hom_R(\H_i(F_\b\otimes_RC),N)\cong \H^i\Hom_R(F_\b,\Hom_R(C,N))
     \]
     for all $i\ge 1$. In particular, $\H^n\Hom_R(F_\b,\Hom_R(C,N))\cong \Hom_R(\H_n(F_\b\otimes_RC),N)$ which is 0 by \Cref{cquasilemma}. Thus we get a surjection $\Hom_R(F_{n-1},\Hom_R(C,N)) \to \Hom_R(F_n,\Hom_R(C,N)) \to 0$. Since $F_\b$ is minimal, Nakayama's Lemma implies that $\Hom_R(C,N)=0$ and hence, $N=0$ which is a contradiction. Finally, let $t\geq 2$. There is an exact sequence $0\to L \to F_0\otimes_RC \to M \to 0$ for some module $L$ (provided one chooses $F_\b$ such that $\inf F_\b=\hinf F_\b$, which one can do by \Cref{lem:hinf=inf}). Then we have $\Ext^{i+1}_R(M,N)\cong\Ext^i_R(L,N)$ for all $i>0$, and $\depth_R L= \depth M + 1$. Thus the assertion follows by Proposition \ref{prop:3.3}, Theorem \ref{thm:AB}, and induction.
\end{proof}

\begin{remark}
When $C=R$ one recovers from \Cref{thm:supext} the recent result \cite[Theorem 1.1(1)]{qpdIschebeck}.
\end{remark}

\begin{corollary}\label{cor:DepFormQPD}
    Let $R$ be a local ring, $C$ a semidualizing $R$-module and let $M$, $N$ be nonzero finitely generated $R$-modules. Assume that
    \begin{enumerate}
    \item $\cqpd{C}{R}M<\infty$ and $M\in\B{C}(R)$,
    \item $N\in\A{C}(R)$,
    \item $\Tor^R_{>>0}(M,N)=0$.
    \end{enumerate}
    Then $M$ and $N$ satisfy the dependency formula.
\end{corollary}

\begin{proof}
By \cite[Lemma 3.1.13(c)]{KeriSD}
\[
\Tor^R_i(M,N)\cong\Tor_R^i(\Hom_R(C,M),C\otimes_RN)\quad\forall i.
\]
By \Cref{thm:cqpd=qpd} $\qpd{R}\Hom_R(C,M)<\infty$, therefore by \Cref{thm:DepFor}
\[
\sup\{i\mid\Tor^R_i(M,N)\neq 0\}=\sup\{\depth R_{\p}-\depth_{R_\p}\Hom_{R_\p}(C_\p,M_\p)-\depth_{R_p}(C_\p\otimes_{R_\p}N_\p)\mid\p\in \Spec R\}.
\]
Now one concludes as in \Cref{cor:DepForCpd}.
\end{proof}

\section{Applications of the $C$-quasi-injective dimension}
In this section we provide some applications of the $C$-quasi-injective dimension. The Proposition below is a generalization of \cite[Proposition 3.4(3)]{qid}, we first prove a $C$-injective version of it as a Lemma.

\begin{lemma}\label{cid}
  Let $\mathfrak{a}$ be an ideal of $R$, and let $N$ be an $R$-module. If $\cid{C}RN=d< \infty$, then $\H^i_{\mathfrak{a}}(N)=0$ for all $i>d$.
\end{lemma}
\begin{proof}
    First we show that for every finitely generated module $M$ and an injective $R$-module $I$, one has $\H^{i>0}_{\mathfrak{a}}(\Hom_R(M,I))=0$. Set $L=\Hom_R(M,I)$ and let $\dots \to F_1 \to F_0 \to M \to 0$ be a free resolution of $M$. By applying $\Hom_R(-,I)$, we get an injective resolution 
    $0\to L \to \Hom_R(F_0,I)\to \Hom_R(F_1,I) \to \dots$ of $L$. Then by applying $\Gamma_\mathfrak{a}(-)$, we have a commutative diagram
$$
\xymatrix{
0\ar[r]& \Gamma_\mathfrak{a}(L)\ar[r]& \Gamma_\mathfrak{a}(\Hom_R(F_0,I))\ar[r]& \Gamma_\mathfrak{a}(\Hom_R(F_1,I))\ar[r]& \dots\\
0\ar[r]& \Hom_R(M,\Gamma_\mathfrak{a}(I))\ar[u]^\cong \ar[r]& \Hom_R(F_0,\Gamma_\mathfrak{a}(I))\ar[r]\ar[u]^\cong& \Hom_R(F_1,\Gamma_\mathfrak{a}(I))\ar[u]^\cong \ar[r]& \dots, }
$$
where the vertical maps are natural isomorphisms; see the proof of \cite[Theorem 3.2]{Zargar}. Since $\Gamma_\mathfrak{a}(I)$ is an injective module, the bottom row is exact. Therefore the upper row is also exact. This shows that $\H^{>0}_\mathfrak{a}(L)=0$ as desired.

Now assume $\cid{C}RN=d< \infty$, and let $I_\b=(0\to I_0 \to I_{-1} \to \dots \to I_{-d}\to 0)$ be a complex of injective modules such that $\Hom_R(C,I_\b)$ is a $C$-injective resolution of $N$. By the last part, we have $\H^{>0}_\mathfrak{a}(\Hom_R(C,I_i))=0$, for all $i$. Hence by \cite[Exercise 4.1.2]{BS}, we have $\H^i_{\mathfrak{a}}(N)\cong \H^i(\Gamma_\mathfrak{a}(\Hom_R(C,I_\b)))=0$ for all $i>d$.
\end{proof}

\begin{proposition}\label{localcoh}
    Let $\cqid{C}RN < \infty$. Then $\H^i_\mathfrak{a}(N)=0$ for all $i>\cqid{C}RN$.
\end{proposition}
\begin{proof}
    By using Lemma \ref{cid}, the same argument as in the proof of \cite[Proposition 3.4(3)]{qid} applies to our case.
\end{proof}

The next Proposition generalizes \cite[Proposition 3.4(2)]{qid}.

\begin{proposition}\label{Vext}
    Let $M$ and $N$ be $R$-modules. Suppose $\cqid{C}{R} N < \infty$ and $\Tor^R_{>0}(C,M)=0$. If $\Ext^{>>0}_R(M,N)=0$, then $\Ext^i_R(M,N)=0$ for all $i>\cqid{C}{R}N$. 
\end{proposition}
\begin{proof}
    Let $F_\b$ be a projective resolution of $M$, and let $I_\b$ be a bounded $C$-quasi-injective resolution of $N$.
    Then the double complex $\Hom_R(F_\b,\Hom_R(C,I_\b))$ induces a third quadrant spectral sequence 
    $$\E^{ij}_2\cong \Ext^i_R(M,\H^j(\Hom_R(C,I_\b))) \Longrightarrow \H^{i+j} \Tot(\Hom_R(F_\b,\Hom_R(C,I_\b))).$$
    Since $\Tor^R_{>0}(C,M)=0$, the canonical isomorphism $\Hom_R(F_\b,\Hom_R(C,I_\b)) \cong \Hom_R(F_\b\otimes_RC,I_\b)$ implies that $\H^{i+j} \Tot(\Hom_R(F_\b,\Hom_R(C,I_\b))) \cong \H^{i+j}\Hom_R(M\otimes_R C, I_\b)$.
    Set $h=\hinf I_\b$, $l=\inf I_\b$, and $n=\sup \{i| \Ext^i_R(M,N)\neq 0\}$. Since $\H^i(\Hom_R(C,I_\b))$ is isomorphic to a direct sum of copies of $N$ and the maps on $\E_2$ page are of bidegree $(2,-1)$ we have $\E^{nh}_\infty\cong\E^{nh}_2\cong \Ext^n_R(M,\H^h(\Hom_R(C,I_\b)))\neq 0$. Thus, one has $n+h\leq l$ and hence $n\leq l-h \leq \cqid{C}{R}M$.
\end{proof}

The next Theorem generalizes \cite[Theorem 3.2]{qid} which was itself a generalization of the Bass' formula.

\begin{theorem}\label{Bassformula}
    Let $R$ be a local ring. Let $M$ be a nonzero finitely generated $R$-module such that $\cqid{C}RM < \infty$. Assume either $R$ is Cohen-Macaulay or $\Tor^R_{>0}(C,M)=0$. Then $\cqid{C}RM = \depth R$. 
\end{theorem}
\begin{proof}
     Let $I_\b$ be a bounded $C$-quasi-injective resolution of $M$ such that  $\cqid{C}RM = \hinf(\Hom_R(C,I_\b))-\inf(\Hom_R(C,I_\b))$. 
     Set $t=\cqid{C}RM$ and $d=\depth R$. Since $\depth_RC=\depth R$, we can choose a maximal regular sequence $x_1,\dots,x_d$ on both $R$ and $C$. By  \Cref{Vext}, $\Ext^i_R(R/(x_1,\dots,x_d),M)=0$ for all $i>t$ while $\Ext^d_R(R/(x_1,\dots,x_d),M)\cong M/(x_1,\dots,x_d)M \neq 0$ by \cite[Proposition 1.6.10]{BrunsHerzog}. Therefore we must have $d \leq t$.
     
     Next, we show $t\leq d$. If $t=0$, we have nothing to prove. Assume $t>0$ and set $s=\hinf \Hom_R(C,I_\b)$, and $Z_s=\ker (\partial^{\Hom_R(C,I_\b)}_s)$. One has $\cqid{C}RM=\cid{C}{R}Z_s$, and  by \cite[Theorem 2.11]{TakWhite}, $\cid{C}{R}Z_s=\id_RC\otimes_RZ_s$. By \cite[Corollary 3.1.12]{BrunsHerzog}, there exists a prime ideal $\p$ such that $\Ext^t_R(R/\p,C\otimes_RZ_s)\neq 0$. Let $g=\grade_R(\p)$. If $R$ is Cohen-Macaulay, then there exists a regular sequence $a_1,\dots,a_g \in \p$ and $\p \in \Ass R/(a_1,\dots,a_g)$. Thus, the exact sequence $0 \to R/\p \to R/(a_1,\dots,a_g)$ induces an exact sequence $\Ext^t_R(R/(a_1,\dots,a_g),C\otimes_RZ_s)\to \Ext^t_R(R/\p,C\otimes_RZ_s) \to 0$ showing that $\Ext^t_R(R/(a_1,\dots,a_g),C\otimes_RZ_s)\neq 0$. Therefore, one has $t\leq g \leq d$. 

     Next, assume $\Tor^R_{>0}(C,M)=0$. There are exact sequences 
     $$\begin{cases}
    0\to Z_j\to \Hom_R(C,I_j)\to B_j\to0\\
    0\to B_{j+1}\to Z_j\to \H_j(\Hom_R(C,I_\b))\to0.
    \end{cases}$$
     Since $\Tor^R_{>0}(C,M)=0$, one checks by induction that $\Tor^R_{>0}(C,B_j)=\Tor^R_{>0}(C,Z_j)=0$, for all $j$.  Hence, by applying $C\otimes_R-$, we get exact sequences 
     $$\begin{cases}
    0\to C\otimes_RZ_j\to C\otimes_R\Hom_R(C,I_j)\to C\otimes_RB_j\to0\\
    0\to C\otimes_RB_{j+1}\to C\otimes_RZ_j\to C\otimes_R\H_j(\Hom_R(C,I_\b))\to0.
    \end{cases}$$
     We have $C\otimes_R\Hom_R(C,I_j)\cong I_j$ and $\id_RC\otimes_RZ_s=t$. Therefore, the same arguments as in \cite[Lemma 3.1 and Theorem 3.2]{qid} show that $t\leq d$.
\end{proof}

\begin{corollary}
    Let $R$ be local. Let $M$ be a nonzero finitely generated $R$-module such that $\cqid{C}RM < \infty$ and $\Tor^R_{>0}(C,M)=0$. If $\dim_RM=\dim R$, then $R$ is Cohen-Macaulay.
\end{corollary}
\begin{proof}
    By Theorem \ref{Bassformula}, $\cqid{C}RM =\depth R$. Using Proposition \ref{localcoh} and Grothendieck's Nonvanishing Theorem, we have $\dim R = \dim_RM\leq \depth R$.
    Thus $R$ is Cohen-Macaulay.
\end{proof}

\begin{question}
    Does \Cref{Bassformula} hold without the hypotheses that $R$ is Cohen-Macaulay or $\Tor^R_{>0}(C,M)=0$?
\end{question}

The next result generalizes Ischbeck's formula \cite{ischebeck} to the $C$-quasi-injective dimension case.

\begin{theorem}\label{thm:supExt}
    Let $(R,\m,k)$ be a local ring, and let $M$, $N$ be finitely generated nonzero $R$-modules. Assume $M\in \A{C}(R)$, $\Tor^R_{>0}(C,N)=0$ and $\cqid{C}{R}N<\infty$. If $\Ext^{>>0}_R(M,N)=0$, then one has   $$\sup\{ i| \Ext^i_R(M,N)\neq 0\} =  \depth R - \depth_R M.$$
\end{theorem}
\begin{proof} 
   Set $t=\depth R - \depth_RM$, and assume $\cqid{C}{R}N<\infty$. We proceed by induction on $\depth_RM$. If $\depth_RM=0$, then $t=\depth R$ and by \Cref{Vext} and \Cref{Bassformula}, we have $\Ext^i_R(M,N)=0$ for all $i>t=\cqid{C}{R}N$. We show $\Ext^t_R(M,N)\neq 0$. We assume to contrary that $\Ext^t_R(M,N)=0$ and seek a contradiction. Let $I_\b$ be a $C$-quai injective resolution of $N$ such that $\cqid{C}{R}N=\hinf(\Hom_R(C,I_\b))-\inf(\Hom_R(C,I_\b))=t$. Set $X_\b=\Hom_R(C,I_\b)$ and $\hinf(X_\b)=-s$. Then, there are exact sequences
    $$\begin{cases}
    0\to Z_j\to X_j\to B_j\to0\\
    0\to B_{j+1}\to Z_j\to \H_j(X_\b)\to0.
    \end{cases}$$
    By using the exact sequences above and induction, one checks $\Ext^i_R(M,B_j)=\Ext^i_R(M,Z_j)=0$ for all $i \geq t$ and all $j\in \mathbb{Z}$. Since by \cite[Theorem 2.11]{TakWhite}, $\id_RC\otimes_RZ_{-s}=\cid{C}{R}Z_{-s}=t$, we have $\Ext^t_R(k,C\otimes_RZ_{-s})\neq 0$; see the proof of \cite[Theorem 3.2]{qid}. Since $M\in \A{C}(R)$, we have $M\cong \Hom_R(C,C\otimes_RM)$ and therefore $\depth_R C\otimes_RM=0$ (by \cite[Lemma 3.9]{AiTak}). There is an exact sequence $0\to k\to C\otimes_RM \to L\to 0$ of $R$-modules. Then by applying $\Hom_R(-,C\otimes_RZ_{-s})$, one gets an exact sequence $\Ext^t_R(C\otimes_RM,C\otimes_RZ_{-s})\to \Ext^t_R(k,C\otimes_RZ_{-s})\to 0$ which shows that $\Ext^t_R(C\otimes_RM,C\otimes_RZ_{-s})\neq 0$. Since $\cid{C}{R}Z_{-s}<\infty$, by \cite[Corollary 2.9(b)]{TakWhite} it follows that $Z_{-s}\in \A{C}(R)$, therefore by \cite[Lemma 3.1.13(a)]{KeriSD}, we deduce that $\Ext^t_R(C\otimes_RM,C\otimes_RZ_{-s})\cong \Ext^t_R(M,Z_{-s})=0$, contradiction \\ 
    Now assume $\depth_RM>0$ and let $x\in \m$ be a nonzero-divisor on $M$. By using the exact sequence $0\to M \overset{x}\to M \to M/xM \to 0$ and \cite[Proposition 3.1.7]{KeriSD}, we have $M/xM \in \A{C}(R)$. There exists a long exact sequence $\dots \to \Ext^i_R(M,N) \overset{x}\to \Ext^i_R(M,N) \to \Ext^i_R(M/xM,N) \to \Ext^{i+1}_R(M,N) \to \cdots$. Hence, Nakayama's Lemma and an induction argument settles the assertion.
\end{proof}

The next result generalizes \cite[Corollary 6.21]{qpd} and recovers \cite[Corollary 4.3]{qid} for rings with a dualizing complex.

\begin{theorem}\label{0.8}
    Let $R$ be a local ring with a dualizing complex $D_R$. Let $C$ be a semidualizing module, and let $M$ be a finitely generated nonzero $R$-module. Assume one of the following conditions hold:
    \begin{enumerate}
        \item $\cqpd{C}{R}M$ and $\id_RM$ both are finite.
        \item $\cqid{C}{R}M$ and $\pd_RM$ both are finite.
    \end{enumerate}
    Then $R$ is Cohen-Macaulay with canonical module $\omega_R \cong C$.
\end{theorem}
\begin{proof}
    Assume (1) holds. Since $M$ is a finitely generated module of finite injective dimension, it follows that $R$ is Cohen-Macaulay and hence $D_R\simeq \omega_R$. By \Cref{minimal} there exists a  perfect complex  $P_\b$ such that $\H_i(P_\b\otimes_RC)\cong M^{\oplus{a_i}}$ for some $a_i$. Since $\id_RM<\infty$, it follows that $\id_R(P_\b\otimes_R C)<\infty$ by \cite[Lemma 4.1]{GheibiZargar}. Thus, one has $\pd_R\RHom_R(P_\b\otimes_R C,\omega_R)<\infty$ by \cite[Remark 4.1]{IK}. By using \cite[Theorem A.7.8]{gbook} and the isomorphism $\RHom_R(P_\b\otimes_R C,\omega_R) \cong \RHom_R(P_\b,\RHom_R(C,\omega_R))$ , we have $\pd_R\RHom_R(C,\omega_R)<\infty$. Since $\RHom_R(C,\omega_R)\cong \Hom_R(C,\omega_R)$ is MCM by \cite[Lemma 3.9]{AiTak}, it follows that $\Hom_R(C,\omega_R)$ is free. Since modules of finite injetive dimensions belong to $\B{C}(R)$ it follows that $C\otimes_R\Hom_R(C,\omega_R)\cong\omega_R$. Therefore $\omega_R\cong C^n$, but since $\omega_R$ is indecomposable we deduce that $\omega_R\cong C$.

    Assume (2) holds. There exists a bounded complex of injective modules $I_\b$ such that $\H_i(\Hom_R(C,I_\b)) \cong M^{\oplus^{a_i}}$ for some $a_i$. Since $\pd_RM<\infty$, we have $\pd_R\Hom_R(C,I_\b)<\infty$ by \Cref{lem:thick}. It follows that $\id_R\RHom(\RHom_R(C,I_\b),D_R)<\infty$ by \cite[Remark 4.1]{IK}. By \cite[A.4.24]{gbook}, there is an isomorphism $\RHom_R(\RHom_R(C,I_\b),D_R) \cong C\lotimes_R\RHom_R(I_\b,D_R)$. Therefore, one has $\id_R C\lotimes_R\RHom_R(I_\b,D_R)<\infty$. Next we show that $\Hom_R(I_\b,D_R) \in \D^f_b(R)$. Indeed, since $\RHom_R(C,I_\b)=\Hom_R(C,I_\b) \in \D^f_b(R)$, we have $\RHom(\RHom_R(C,I_\b),D_R) \in \D^f_b(R)$ by \cite[Lemma A.4.4]{gbook}, and hence, $C\lotimes_R\RHom_R(I_\b,D_R)\in \D^f_b(R)$. Therefore $\RHom_R(C,C\lotimes_R\RHom_R(I_\b,D_R))\in \D^f(R)$ by a further application of \cite[Lemma A.4.4]{gbook}. Since $\RHom_R(I_\b,D_R)$ is a bounded complex of flat modules, it follows from \cite[A.4.23]{gbook} that
    \[
\RHom_R(C,C\lotimes_R\RHom_R(I_\b,D_R))\cong\RHom_R(C,C)\lotimes_R\RHom_R(I_\b,D_R)\cong\RHom_R(I_\b,D_R)\in\D^f_b(R).
    \]
    Since $\pd_R\RHom_R(I_\b,D_R)<\infty$ by \cite[Remark 4.1]{IK}, there exists a perfect complex $F_\b$ and a quasi isomorphism $F_\b \overset{\simeq} \longrightarrow \RHom_R(I_\b,D_R)$. Since we already proved that $\id_R C\lotimes_R\RHom_R(I_\b,D_R)<\infty$, it follows that $\id_R C\otimes_R F_\b <\infty$. By \cite[Remark 4.1]{IK} $\pd_R\Hom_R(C\otimes_R F_\b, D_R)=\pd_R\Hom_R(F_\b,\Hom_R(C,D_R))<\infty$. Since both $F_\b$ and $\Hom_R(C,D_R)$ are in $\D^f_b(R)$, we get $\pd_R\Hom_R(C,D_R)<\infty$ by \cite[Theorem A.7.8]{gbook}. Finally, by \cite[Theorem A.8.4]{gbook} the following isomorphism $\RHom_R(\RHom_R(C,D_R),D_R)\cong C$ holds in the derived category of $R$, showing that $\id_RC<\infty$. Thus $R$ is Cohen-Macaulay with a canonical module $\omega_R=C$.
\end{proof}

The next result is an dual version of \Cref{thm:ARconj}.

\begin{theorem}
Let $M$ and $C$ be $R$-modules where $C$ is semidualizing. If $\cqid{C}{R}M<\infty$ and $\Ext^{\geq2}_R(M,M)=0$, then $\cid{C}{R}M<\infty$. Moreover, if $M$ is finitely generated and $R$ is local, then $\cid{C}{R}M=\depth R$.
\end{theorem}

\begin{proof}
There is a bounded complex $I_\b$ of injective $R$-modules such that $\Hom_R(C,I_\b)$ is not acyclic and its homologies are finite direct sum of copies of $M$. Therefore
\[
\Ext_R^{\geq2}(\H_i(\Hom_R(C,I_\b)),\H_j(\Hom_R(C,I_\b)))=0
\]
for all $i,j\in\mathbb{Z}$. By \cite[Tag 0GM4]{stacks-project} we get $\Hom_R(C,I_\b)=\bigoplus_{i\in\mathbb{Z}}\Sigma^iM^{a_i}$ for some non-negative integers $a_i$, at least one of which is non-zero. Applying $-\lotimes_RC$ to both sides and using \cite[A.4.23]{gbook} we get $I_\b=\bigoplus_{i\in\mathbb{Z}}\Sigma^i(M\lotimes_RC)^{a_i}$. Thus $M\lotimes_RC$ has finite injective dimension and therefore $\cid{C}{R}M<\infty$ by \cite[Definition 3.1(1)]{totushek}. The last assertion follows from \cite[Theorem 2.11(b)]{TakWhite}.
\end{proof}

\section{Declarations}
\textbf{Conflict of Interest Statement.} The authors of this manuscript declare they have no conflict of interest.

\bibliographystyle{amsplain}
\bibliography{biblio}
\end{document}